\documentclass[a4paper]{siamart220329}

\usepackage{amsmath, latexsym, amsfonts, amssymb, amscd,bbm}
\usepackage[english]{babel}
\usepackage{mathrsfs,stmaryrd}

\usepackage{color}

\newsiamthm{hypothesis}{Hypothesis}
\newsiamthm{claim}{Claim}
\newsiamremark{remark}{Remark}
\crefname{hypothesis}{Hypothesis}{Hypotheses}

\newcommand{\norm}[1]{\left \| #1 \right \|}

 \newcommand{\lsup}[1]{\underset{#1\to\infty}{\overline{\lim}}}

\title{The Kinetic Limit of Balanced Neural Networks}
\author{  James MacLaurin   \thanks{Department of Mathematical Sciences. 
    New Jersey Institute of Technology, \email{james.maclaurin@njit.edu}
  }
  \and
  Pedro Vilanova  \thanks{Stevens Institute of Technology, \email{pedro.vilanova@gmail.com}
  } 
  }

\begin{document}
\maketitle

\begin{abstract}
The theory of `Balanced Neural Networks' is a very popular explanation for the high degree of variability and stochasticity in the brain's activity. Roughly speaking, it entails that typical neurons receive many excitatory and inhibitory inputs. The network-wide mean inputs cancel, and one is left with the stochastic fluctuations about the mean. In this paper we determine kinetic equations that describe the population density. The intrinsic dynamics is nonlinear, with multiplicative noise perturbing the state of each neuron. The equations have a spatial dimension, such that the strength-of-connection between neurons is a function of their spatial position. Our method of proof is to decompose the state variables into (i) the network-wide average activity, and (ii) fluctuations about this mean. In the limit, we determine two coupled limiting equations. The requirement that the system be balanced yields implicit equations for the evolution of the average activity.  In the large $n$ limit, the population density of the fluctuations evolves according to a Fokker-Planck equation. If one makes an additional assumption that the intrinsic dynamics is linear and the noise is not multiplicative, then one obtains a spatially-distributed `neural field' equation.
\end{abstract}

\section{Introduction}

In theoretical neuroscience, it is widely conjectured that neurons are typically dynamically balanced, with a high number of excitatory and inhibitory inputs \cite{Shadlen1994,Tsodyks1995,Amit1997,Vreeswijk1996,VanVreeswijk1998,Shadlen1998,Engelken2022}. It is thought that the dynamic balance could explain the high degree of stochasticity and variability in cortical discharge, which is indicated by the fact that the coefficient of variation in cortical spike trains is typically $O(1)$ \cite{Softky1993}. Roughly speaking, the theory is that the mean excitation and inhibition approximately `cancel', and what is left are the stochastic fluctuations about the mean \cite{Shadlen1994,Amit1997}.  

Early work by Van Vreeswijk and Sompolinsky \cite{VanVreeswijk2003} determined conditions for the existence of a balanced state by averaging over all times, and all realizations of the network. Variants of this method were also employed by Rosenbaum and Doiron \cite{Rosenbaum2014}, and Darshan, Hansel and Van Vreeswijk \cite{Darshan2018}. The approach of this article is more geared towards a dynamical systems approach: that is, we wish to determine conditions under which there exists an autonomous flow operator that describes the time-evolution of the network (in the large size limit). This approach is particularly useful for studying situations for which the system is out-of-equilibrium, or for which there does not exist a unique globally-attracting fixed point (such as if there is a limit cycle, or there is multistability).  

This balanced paradigm has proved extremely popular and has been explored in numerous directions. Some applications include: explaining oscillations and rhythms in brain activity \cite{Brunel2000}, UP / DOWN transitions \cite{Tartaglia2017}, working memory models \cite{BrunelWang2001Working,Lim2014,Rubin2017}, pattern formation and spatially-distributed neural activity \cite{Litwin-Kumar2012,Rosenbaum2014,Rosenbaum2017}. Other applications have explored how balanced networks can process sensory cues \cite{Toyoizumi2024}. Pehlevan and Sompolinsky analyze how sparse balanced networks respond selectively to inputs \cite{Pehlevan2014}. Hansel and Mato determined a range of bifurcation in balanced networks of excitation and inhibition \cite{Hansel2002}. Monteforte and Wolf \cite{Monteforte2010} computed the Lyapunov Exponents of sparse balanced networks, obtained a precise understanding of the chaotic nature of the networks. Boerlin, Machens and Deneve analyze how balanced networks can represent information in their spikes \cite{Boerlin2013}. Kadmon \cite{Kadmon2020} examines how balanced networks can enable predictive coding.

We determine the kinetic (large $n$ limit) of a network with many similar characteristics to the original model of Van Vreeswijk and Sompolinsky \cite{Vreeswijk1996,VanVreeswijk1998,VanVreeswijk2003}. It is an all-to-all network, consisting of $n \gg 1$ excitatory neurons and $n \gg 1$ inhibition neurons. The excitatory neurons excite all of the other neurons, and the inhibitory neurons inhibit all of the other inhibitory neurons. The neurons are embedded in a manifold, and the strength-of-interaction depends on their respective positions (this is a reasonably common assumption, see for instance \cite{Bressloff2012,Avitabile2024}). Each neuron is perturbed by multiplicative white noise, which primarily represents anomalous inputs from other parts of the brain. The interaction strength is scaled as $n^{-1/2}$, so that the sum of the absolute values of all of the inputs to any particular neuron diverges with $n$. However under conditions to be outlined further below, the system does not blow up, because the excitatory and inhibitory inputs balance. 



The scaling of the interaction is different from the standard $n^{-1}$ scaling for particle systems with weak interactions (i.e. for Mckean-Vlasov systems \cite{Sznitman1989,Baladron2012,DeMasi2014,Jabin2017,Lucon2020a,Chaintron2021,Bramburger2023,Avitabile2024} or high-dimensional Poissonian chemical reaction networks \cite{Anderson2015,Chevallier2017,Agazzi2018,AgatheNerine2022,Avitabile2024_2}). The scaling factor is $n^{-1/2}$ (so the effect of one neuron on another is relatively stronger than in the Mckean-Vlasov case). To obtain a hydrodynamic limit, we make use of the fact that the `balanced state' is strongly attracting. This effectively damps down the $O(n^{-1/2})$ terms that could cause blowup, and we are left with the fluctuations about the mean \cite{Katzenberger1991}. See for instance the textbook by Berglund and Gentz for an overview of methods for studying stochastic systems near strongly-attracted manifolds \cite{Berglund2006}, or the recent papers \cite{Parsons2017a,Adams2025} that study the quasi-steady distribution for stochastic systems near attracting manifolds over long timescales.  

There are several other recent works that have explored the effect of inhibition on mean-field interacting particle systems. Erny, Locherbach and Loukianova consider interacting Hawkes Processes in the diffusive regime (so that interactions are scaled by $(n^{-/2})$). In their model the kinetic limit always remains balanced because the excitatory and inhibitory neurons have the same dynamics \cite{Erny2021}. By contrast in this paper, the different timescales and dynamics associated to the excitatory and inhibitory neurons leads to a nontrivial algebraic identity for the evolution of the mean activity. Pfaffelhuber, Rotter and Stiefel \cite{Pfaffelhuber2022} consider a system of Hawkes Process that is similar to that of \cite{Erny2021}, and is in the balanced regime. Duval, Lucon and Pouzat \cite{Duval2022} consider a network of Hawkes Processes with multiplicative inhibition. In this work the effect of one neuron on another scales as $O(n^{-1})$.

In an earlier paper \cite{MacLaurin2025Balanced} we determined the kinetic limit of a different high-dimensional balanced network. In \cite{MacLaurin2025Balanced}, the intrinsic synaptic dynamics is linear, and this facilitated the analysis because the dynamics could be decomposed into equations for the mean and variance. By contrast, in this paper the intrinsic  dynamics is nonlinear, and we must determine a new means of decomposing the dynamics into a projection onto the balanced manifold, and nonlinear fluctuations about it. To the best of the knowledge of the authors, \cite{MacLaurin2025Balanced} and this paper represent the first rigorous derivation of the kinetic limit of balanced neural networks.

A very recent preprint of Quininao and Touboul \cite{quininao2025balanced} performs several numerical simulations of a set of balanced SDEs very similar to this paper. Their simulations reveal that in many situations a balanced state exists. They conjecture that an algebraic criterion that is similar to this paper must be satisfied in a balanced state. Perhaps the most salient difference between the conjectures of \cite{quininao2025balanced} and the results of this paper is that in order that we may prove that the population density concentrates at a unique distribution as $n\to\infty$, we also require that the system-wide mean excitation / inhibition is stable to small perturbations. It is hard to see how this assumption could be significantly weakened, because if there were no linear stability, then small stochastic perturbations would be enormously amplified by the strong interaction.

It is also worth comparing these equations to the `spin-glass' dynamical models \cite{Sompolinsky1981a,Crisanti1993,BenArous1995,BenArous2006,Faugeras2015,Crisanti2018,Helias2019,MacLaurin2024}. These models also have $n^{-1/2}$ scaling of the interactions. However, the interactions themselves are mediated by static Gaussian random variables, of zero mean and unit variance. These models have also been heavily applied to neuroscience \cite{Moynot2002,Faugeras2015,Parisi2023}. One of the most important differences is that in the spin glass model, an individual neuron has both excitatory and inhibitory effects on other neurons, which violates Dale's Law \cite{Bhuiyan2022}. However, in the balanced model of this paper, individual neurons are either purely excitatory, or purely inhibitory. Indeed, the hydrodynamic limiting equations are different (one can compare the limiting equations of this paper to for instance the equations in \cite{Crisanti2018,MacLaurin2024}).


\textit{Notation:}

For any Polish Space $\mathcal{X}$, we let $\mathcal{C}(\mathcal{X})$ denote all continuous functions $\mathcal{X} \mapsto \mathbb{R}$ and we let $\mathcal{P}(\mathcal{X})$ denote the set of all probability measures on $\mathcal{X}$. $\mathcal{P}(\mathcal{X})$ is always endowed with the topology of weak convergence, i.e. generated by open sets of the form, for any continuous bounded function $f \in \mathcal{C}_b(\mathcal{X})$ and $\epsilon > 0$,
\[
\bigg\lbrace \nu \in \mathcal{P}(\mathcal{X}) \; : \; \big| \mathbb{E}^{\nu}[f] - c \big| < \epsilon  \bigg\rbrace .
\]
Let $d_W$ be the Wasserstein Metric on $\mathbb{R}^2$. That is, for any $\mu,\nu \in \mathcal{P}\big( \mathbb{R}^2 \big)$,
\begin{align}
d_W(\mu,\nu) = \inf\bigg\lbrace \mathbb{E}^{\zeta}\big[ |y_e - z_e| + |y_i - z_i|  \big] \bigg\rbrace
\end{align}
where the infimum is over all couplings such that the law of $(y_e,y_i)$ is $\mu$, and the law of $(z_e,z_i)$ is $\nu$.

\section{Outline of Model and Assumptions}
\label{Section Assumotions}
We consider a balanced network of SDEs. There are $n \gg 1$ excitatory neurons and $n \gg 1$ inhibitory neurons. It is assumed that the neurons reside in a compact domain $\mathcal{E} \subseteq \mathbb{R}^d$. The $j^{th}$ excitatory and $j^{th}$ inhibitory neuron are each assigned a position $x^j_n \in \mathcal{E}$.

First, we require that the spatial distribution of the neurons throughout $\mathcal{E}$ converges. Let $\hat{\mu}^n(x) = n^{-1} \sum_{j\in I_n} \delta_{x^j_n} \in \mathcal{P}(\mathcal{E})$ denote the empirical measure generated by the positions.
\begin{hypothesis}\label{Hypothesis Distribution}
 It is assumed that there exists a measure $\kappa \in \mathcal{P}\big( \mathcal{E} \big)$ such that
\begin{align}
\lim_{n\to\infty} \hat{\mu}^n(x)  = \kappa.
\end{align}
\end{hypothesis}
The average connectivity strength is indicated by continuous functions $\big\lbrace \mathcal{K}_{\alpha\beta} \big\rbrace_{\alpha,\beta \in \lbrace e,i \rbrace} \subset \mathcal{C}\big( \mathcal{E} \times \mathcal{E} \big)$. We make the following finite-rank assumption.
\begin{hypothesis}
There exists a positive integer $M > 0$ and continuous functions $\lbrace h_i \rbrace_{1\leq i \leq M} \subset \mathcal{C}(\mathcal{E})$ and constants $\lbrace c_{\alpha\beta,ij} \rbrace_{i,j \leq M \fatsemi \alpha,\beta \in \lbrace e,i \rbrace} \subset \mathbb{R} $ such that 
\begin{align}
\mathcal{K}_{\alpha\beta}( x, x') = \sum_{i,j=1}^M c_{\alpha\beta,ij}  h_i(x) h_j(x').
\end{align}
The basis functions are such that
\begin{align}
\int_{\mathcal{E}} h_i(x) h_j(x)  d\kappa(x) = \delta_{ij}.
\end{align}
\end{hypothesis}
We let $\mathcal{C}_M(\mathcal{E} \times \mathcal{E})$ be the set of all functions $K$ of the form
\begin{equation}
K(x,x') =  \sum_{i,j=1}^M c_{ij}  h_i(x) h_j(x')
\end{equation}
for some constants $\lbrace c_{ij} \rbrace_{i,j \leq M}$. We let $\mathcal{C}_M(\mathcal{E})$ be the set of all functions of the form
\[
g(x) = \sum_{i=1}^M a_i h_i(x)
\]
for some constants $\lbrace a_i \rbrace_{i\leq M}$.
\begin{remark}
 Let us underscore that these assumptions are consistent with a `mean-field' model with no spatial extension. One simply takes $\mathcal{E}$ to consist of a single point $0$, and take $M=1$ and $\mathcal{K}_{\alpha\beta}(0,0) := 1$.   
\end{remark}

The state of the $j^{th}$ excitatory neuron is written as $z^j_{e,t} \in \mathbb{R}$, and the state of the $j^{th}$ inhibitory neuron is written as $z^j_{i,t} \in \mathbb{R}$. The dynamics is assumed to be of the form
\begin{align}
dz^j_{e,t}  =& \bigg\lbrace  f_e(z^j_{e,t}) + n^{-1/2} \sum_{k\in I_n}\bigg(  \mathcal{K}_{ee}(x^j_n , x^k_n) G_{ee}(z^k_{e,t}) - \mathcal{K}_{ei}(x^j_n , x^k_n)  G_{ei}(z^k_{i,t}) \bigg) \bigg\rbrace dt \nonumber \\ &+ \sigma_e(x^j_n , z^j_{e,t}) dW^j_{e,t} \\
dz^j_{i,t} =& \bigg\lbrace f_i(z^j_{i,t}) + n^{-1/2} \sum_{k\in I_n}\bigg( \mathcal{K}_{ie}(x^j_n , x^k_n) G_{ie}(z^k_{e,t}) -\mathcal{K}_{ii}(x^j_n , x^k_n)  G_{ii}(z^k_{i,t}) \bigg) \bigg\rbrace dt  \nonumber \\ &+ \sigma_i(x^j_n , z^j_{i,t}) dW^j_{i,t} .
\end{align}
Here $\lbrace W^j_{e,t} , W^j_{i,t} \rbrace_{j\in I_n}$ are independent Brownian Motions. The initial conditions $\lbrace z^j_{e,0} , z^j_{i,0} \rbrace_{j\in I_n}$ are constants. We next make some assumptions on the regularity of the functions.

\begin{hypothesis} 
\begin{itemize}
\item It is assumed that the functions $ \big\lbrace f_{\alpha} , G_{\alpha\beta}  \big\rbrace_{\alpha,\beta \in \lbrace e,i \rbrace} \subset \mathcal{C}^2(\mathbb{R})$ are twice continuously differentiable, with all derivatives upto second order uniformly bounded. Also $\sigma_e$ and $\sigma_i$ are twice continuously differentiable
\item It is assumed that $|\sigma_e|$ and $|\sigma_i|$ are uniformly bounded by a constant $C_{\sigma} > 0$.
\end{itemize}
\end{hypothesis}

\subsection{Balanced Assumptions}
We are going to see that, as $n\to\infty$, the dynamics is pulled towards a balanced state. To make precise sense of this, we will define a `Balanced Manifold'. Roughly-speaking, the manifold will consist of all systems such that (i) excitation balances inhibition, and (ii) it is stable to perturbations in the mean excitation / inhibition throughout the system. We first require some additional definitions.

Let $\mathcal{P}_*\big( \mathcal{E} \times \mathbb{R}^2 \big)$ denote all probability measures $\mu$ such that for all $a\leq M$,
 \begin{align}
 \mathbb{E}^{(x,y_e,y_i) \sim \mu} \big[  h_a(x) y_e \big] &=  \mathbb{E}^{(x,y_e,y_i) \sim \mu} \big[ h_a(x) y_i \big] = 0 \text{ and }\\
  \mathbb{E}^{(x,y_e,y_i) \sim \mu} \big[   y_e^2 + y_i^2 \big] &< \infty.
 \end{align}
   Let $\mathcal{P}_*\big(\mathcal{E} \times \mathcal{C}([0,T],  \mathbb{R}^2) \big)$ denote all measures $\mu$ such that for all $t\leq T$, the marginal at time $t$ is in $\mathcal{P}_*\big( \mathcal{E} \times \mathbb{R}^2 \big)$. 
   
For $\alpha \in \lbrace e,i \rbrace$ and $a \leq M$, define the function $\mathcal{G}^a_{\alpha}:  \mathcal{C}_M(\mathcal{E})^2 \times \mathcal{P}_*\big( \mathcal{E} \times \mathbb{R}^2 \big) \mapsto \mathbb{R}$ to be such that
 \begin{align}
 \mathcal{G}_{\alpha}^a(v, \mu) =& \int_{\mathcal{E}} h_a(z) \mathbb{E}^{(x,y_e,y_i) \sim \mu} \bigg[   \mathcal{K}_{\alpha e}(z,x)  G_{\alpha e}(y_{e} + v_e(x) )  -  \mathcal{K}_{\alpha i}(z,x)  G_{\alpha i}(y_{i} + v_i(x) )  \bigg] d\kappa(z) .
 \end{align} 
 We write $\mathcal{G}:  \mathcal{C}_M(\mathcal{E})^2 \times \mathcal{P}_*\big( \mathcal{E} \times \mathbb{R}^2 \big) \mapsto \mathbb{R}^{2M}$ to be such that
 \begin{align}
 \mathcal{G}(v,\mu) =  \big( \mathcal{G}_{\alpha}^a(v, \mu)  \big)_{a\leq M , \alpha\in \lbrace e,i \rbrace}.
 \end{align}
  Let $\mathcal{J}: \mathcal{C}_M(\mathcal{E})^2 \times \mathcal{P}_*\big( \mathcal{E} \times \mathbb{R}^2 \big) \mapsto \mathbb{R}^{2M \times 2M}$ be the Jacobian of the map $v \mapsto \mathcal{G}(v,\mu)$. That is, $\mathcal{J} = \big( \mathcal{J}^{ab}_{\alpha\beta} \big)_{\alpha,\beta \in \lbrace e,i\rbrace \fatsemi a,b \leq M}$, where for $a,b \leq M$ and $\alpha \in \lbrace e,i \rbrace$, define
  \begin{align}
  \mathcal{J}^{ab}_{\alpha e}(v,\mu) =& \lim_{\epsilon \to 0^+} \epsilon^{-1} \big( \mathcal{G}^a_{\alpha}(v_{e} + \epsilon h_b, v_i , \mu) - \mathcal{G}^a_{\alpha}(v_e,v_i,\mu) \big) \\
    \mathcal{J}^{ab}_{\alpha i}(v,\mu) =& \lim_{\epsilon \to 0^+} \epsilon^{-1} \big( \mathcal{G}^a_{\alpha}(v_{e} , v_i + \epsilon h_b , \mu) - \mathcal{G}^a_{\alpha}(v_e,v_i,\mu) \big).
  \end{align} 
  We notice that, since $G_{\alpha \beta}$ is differentiable,
   \begin{align}
  \mathcal{J}^{ab}_{\alpha e}(v_e,v_i,\mu) &= \int_{\mathcal{E}} h_a(x') \mathbb{E}^{(x,y_e,y_i) \sim \mu} \bigg[  \mathcal{K}_{\alpha e}(x',x)  \dot{G}_{\alpha e}(y_{e} + v_e(x) ) h_b(x)  \bigg] d\kappa(x') \label{eq: Jacobian e} \\
    \mathcal{J}^{ab}_{\alpha i}(v_e,v_i,\mu) &= - \int_{\mathcal{E}} h_a(x') \mathbb{E}^{(x,y_e,y_i) \sim \mu} \bigg[   \mathcal{K}_{\alpha i}(x',x)   \dot{G}_{\alpha i}(y_{i} + v_i(x) ) h_b(x)  \bigg] d\kappa(x') .\label{eq: Jacobian i} 
  \end{align}  
  Define the \textit{Balanced Manifold} $\mathcal{B} \subset \mathcal{C}_M(\mathcal{E} )^2 \times \mathcal{P}_*\big( \mathcal{E} \times \mathbb{R}^2 \big)$, to consist of all $(v_e,v_i,\mu)$ such that (i) $\mathcal{G}(v,\mu) = 0$ identically, and (ii) every eigenvalue of $\mathcal{J}(v_e,v_i,\mu)$ has real part that is strictly negative, and (iii) the marginal of $\mu$ over its first variable is $\kappa$. 
  
  Here and below, we need to decompose $z_{\alpha,t}$ into its component in $\mathcal{C}_M(\mathcal{E}  )$, and the component orthogonal to it. To this end, define the projection matrix $Q \in \mathbb{R}^{M\times M}$ to have elements
  \begin{align}
  Q_{pq} = n^{-1}\sum_{j \in I_n} h_p(x^j_n) h_q(x^j_n).
  \end{align}
  It follows from Hypothesis \ref{Hypothesis Distribution} that
  \begin{align}
  \lim_{n\to\infty} Q = \mathbf{I}.
  \end{align}
  We will thus assume throughout this paper that $n$ is large enough that 
  \begin{equation}
  \det(Q) > 1/2.
  \end{equation}
  Next define 
  \begin{align}
  v^a_e(t) =& n^{-1} \sum_{k\in I_n} \sum_{b = 1}^{M} Q^{-1}_{ab}  z^k_{e,t} h_b(x^k_{n})  \\
    v^a_i(t) =& n^{-1} \sum_{k\in I_n} \sum_{b = 1}^{M} Q^{-1}_{ab} z^k_{i,t} h_b(x^k_{n}) 
  \end{align}
  and write $v(t) = \big( v^p_\alpha(t) \big)_{\alpha\in \lbrace e,i \rbrace , p \leq M} \in \mathbb{R}^{2M}$. Define
  \begin{align}
  y^j_{e,t} &= z^j_{e,t} -  \sum_{a = 1}^M v^a_e(t) h_a(x^j_n) \\
    y^j_{i,t} &= z^j_{i,t} -  \sum_{a = 1}^M v^a_i(t) h_a(x^j_n).
  \end{align}
  Notice that for any $a \leq M$ and $\alpha\in \lbrace e,i\rbrace$,
  \begin{align}
 \sum_{j\in I_n} h_a(x^j_n)    y^j_{\alpha,t} = 0.
  \end{align}
  Define the empirical measure
  \begin{align}
  \hat{\mu}^n = n^{-1}\sum_{j\in I_n} \delta_{x^j_n , y^j_{e} , y^j_i} \in \mathcal{P}_*\big( \mathcal{E} \times \mathcal{C}\big( [0,T] , \mathbb{R}^2 \big) \big)
  \end{align}
  and write the marginal at time $t$ as
    \begin{align}
  \hat{\mu}^n_t = n^{-1}\sum_{j\in I_n} \delta_{x^j_n , y^j_{e,t} , y^j_{i,t}} \in \mathcal{P}_*\big( \mathcal{E} \times  \mathbb{R}^2   \big).
  \end{align}
 Next, we require that the initial empirical measure converges to be on the Balanced Manifold.
  \begin{hypothesis} \label{Hypothesis Initial Convergence}
  There exists a measure $\mu_0 \in \mathcal{P}_*(\mathcal{E} \times \mathbb{R}^2)$ such that
  \begin{align}
  \lim_{n\to\infty} \hat{\mu}^n_0 = \mu_0
  \end{align}
and
  \begin{align}
  \mathbb{E}^{\mu_0}\big[ y_{e}^2 + y_i^2 \big] < \infty.
  \end{align}  
  Write $\mu_{0,x} \in \mathcal{P}\big( \mathbb{R}^2 \big)$ to be the law of $(y_e,y_i)$, conditioned on $x$. It is assumed that $\mu_{0,x}$ has a density $p_{0,x} \in L^2(\mathbb{R}^2)$. Also, there exists $\big( \bar{v}_{e}(0), \bar{v}_{i}(0) \big) \in \mathbb{R}^2$ such that
  \begin{align}
  \lim_{n\to\infty} \big| v_{e}(0) - \bar{v}_{e}(0) \big| &= 0 \\
    \lim_{n\to\infty} \big| v_{i}(0) - \bar{v}_{i}(0) \big| &= 0,
  \end{align}
  and it holds that
  \begin{align}
  \big( \bar{v}_e(0) , \bar{v}_i(0) , \mu_0 \big) \in \mathcal{B}.
  \end{align}
  \end{hypothesis}
  
  \section{Main Result}
 We start by outlining the limiting population average and probability distribution. 
  
  \begin{lemma} \label{Lemma Existence of Hydro Limit}
 There exists $ \bar{v}_e , \bar{v}_i \in \mathcal{C}\big([0,T] , \mathcal{C}_M(\mathcal{E}) \big)$ and $\big\lbrace \mu_{x} \big\rbrace_{x\in \mathcal{E}}$ with the following properties. First,
 \begin{equation}
x\mapsto \mu_x \text{ is continuous.}
\end{equation}
 For each $x\in \mathcal{E}$, $\mu_{x} \in \mathcal{P} \big( \mathcal{C}([0,T], \mathbb{R}^2) \big)$ is the probability law of stochastic processes $(y_{e,x} , y_{i,x})$, that are solutions to the following SDEs
 \begin{align}
 dy_{e,x}(t) =& \bigg\lbrace  f_e\big(  y_{e,x}(t) + \bar{v}_{e}(x,t) \big) -\sum_{a=1}^M h_a(x) \int_{\mathcal{E}} h_a(x') \mathbb{E}^{u \sim \mu_{x'}}\big[ f_e\big( u_{e,t}+ \bar{v}_{e}(x',t)  \big) \big] d\kappa(x')  \bigg\rbrace dt \nonumber \\ &+ \sigma_e\big( x,  y_{e,x}(t) + \bar{v}_{e}(x,t) \big)dW_{e,t}   \label{eq: limiting y e x t} \\
  dy_{i,x}(t) =& \bigg\lbrace  f_i \big(  y_{i,x}(t) + \bar{v}_{i}(x,t) \big)  -\sum_{a=1}^M  h_a(x) \int_{\mathcal{E}} h_a(x') \mathbb{E}^{u  \sim \mu_{x'}}\big[ f_i\big( u_{i,t} + \bar{v}_{i}(x',t) \big) \big] d\kappa(x')  \bigg\rbrace dt  \nonumber \\ & + \sigma_i\big( x,  y_{i,x}(t) + \bar{v}_{i}(x,t) \big) dW_{i,t} . \label{eq: limiting y i x t} 
 \end{align}
 The law of $(y_{e,x}(0), y_{i,x}(0))$ is $\mu_{0,x} \in \mathcal{P}\big( \mathbb{R}^2 \big)$. Define $\mu \in \mathcal{P} \big( \mathcal{E} \times \mathcal{C}([0,T], \mathbb{R}^2) \big)$ to be such that for measurable $A \subseteq \mathcal{E}$ and measurable $B \subseteq \mathcal{C}([0,T], \mathbb{R}^2)$,
  \begin{align}
 \mu\big( A \times B \big) = \int_A \mu_x(B) \kappa(dx).
 \end{align}
Let $\mu_t \in \mathcal{P}\big( \mathcal{E} \times \mathbb{R}^2 \big)$ be the marginal of $\mu$ at time $t$. There exists $\eta > 0$ such that for all $t < \eta$,
  \begin{equation}
 \big( \bar{v}_e(t) , \bar{v}_i(t) , \mu_t \big) \in \mathcal{B}.
 \end{equation} 
We take $\eta \in [0,\infty ]$ to be as large as possible. The above specification of $(\bar{v}_e, \bar{v}_i , \mu)$ is unique for all times less than or equal to $\eta$.
  \end{lemma}
The main result of this paper is that $\big( v_e(t) , v_i(t) , \hat{\mu}^n_t \big)_{t < \eta}$  concentrate at $\big(\bar{v}_e(t) , \bar{v}_i(t), \mu_t\big)$ as $n\to\infty$.
\begin{theorem} \label{Theorem Bound Empirical Measure Convergence}
For any $T < \eta$, $\mathbb{P}$-almost-surely
\begin{align}
\lim_{n\to\infty}\bigg\lbrace  \sup_{t\leq T} d_W\big( \hat{\mu}^n_t , \mu_t \big) + \sup_{a\leq M} \sup_{t\leq T} \big| v_{e}^a(t) - \bar{v}_e^a(t) \big| + \sup_{a\leq M} \sup_{t\leq T} \big| v_{i}^a(t) - \bar{v}_i^a(t) \big|  \bigg\rbrace = 0.
\end{align}
where
\begin{equation}
\bar{v}_\alpha^a(t)  = \int_{\mathcal{E}} h_a(x) \bar{v}_\alpha(x,t) d\kappa(x). 
\end{equation}
\end{theorem}
We next outline a specific autonomous expression for the dynamics of $\big( \bar{v}_e(t) , \bar{v}_i(t) , \mu_t \big)_{t < \eta}$.  Write $\mathcal{L}( \bar{v}_e, \bar{v}_i , \mu  )  := \big(   \mathcal{L}^{ab}_{\alpha \beta} ( \bar{v}_e, \bar{v}_i , \mu  )  \big)_{\alpha,\beta \in \lbrace e,i \rbrace , a,b \leq M } \in \mathbb{R}^{2M \times 2M}$ to be the matrix inverse of $ \mathcal{J}(\bar{v}_e  , \bar{v}_i  , \mu  )$.
 \begin{corollary} \label{Corollary Autonomous Dynamics}
 For $1\leq a \leq M$, $\alpha \in \lbrace e,i \rbrace$, and all $t < \eta$,
 \begin{align}
 \frac{d\bar{v}^a_{\alpha}}{dt} = - \sum_{b=1}^M \sum_{\beta\in \lbrace e,i \rbrace} \mathcal{L}^{ab}_{\alpha\beta} ( \bar{v}_e(t), \bar{v}_i(t) , \mu_t ) \mathcal{H}_{\beta}^b( \bar{v}_e(t), \bar{v}_i(t) , \mu_t ) 
 \end{align}
where writing $v_{\alpha}(x) := \sum_{a=1}^M h_a(x) \bar{v}^a_{\alpha}(t)$,
  \begin{multline}
\mathcal{H}_{\alpha}^a(\bar{v}_e, \bar{v}_i , \mu) =  \sum_{b=1}^M  \mathbb{E}^{(x,y_{e}, y_{i}) \sim \mu}\bigg[ h_b(x) \bigg\lbrace c_{\alpha e,ab} \dot{G}_{\alpha e} \big( y_{e} + \bar{v}_{e}(x) \big) \big\lbrace f_e\big( y_{e} + \bar{v}_{e}(x)  \big) - \bar{f}_{e,x}( \bar{v}_e, \bar{v}_i , \mu) \big\rbrace   \\  -  c_{\alpha i,ab}  \dot{G}_{\alpha i} \big( y_{i} + \bar{v}_{i}(x) \big) \big\lbrace f_i\big( y_{i} + \bar{v}_{i}(x)  \big) - \bar{f}_{i,x}(\bar{v}_e, \bar{v}_i , \mu) \big\rbrace    \\ 
+  \frac{1}{2} c_{\alpha e,ab} \ddot{G}_{\alpha e} \big( y_{e} + \bar{v}_{e}(x) \big) \sigma_e^2\big( x,  y_{e} + \bar{v}_{e}(x)  \big) -  \frac{1}{2} c_{\alpha i,ab} \ddot{G}_{\alpha i} \big( y_{i} + \bar{v}_{i}(x) \big) \sigma_i^2\big( x,  y_{i} + \bar{v}_{i}(x) \big) \bigg\rbrace
\bigg] , \label{eq: autonomous v expression}
 \end{multline} 
 where for $\alpha\in\lbrace e,i\rbrace$ and $z \in \mathcal{E}$, $ \bar{f}_{\alpha,z}: \mathcal{C}_M(\mathbb{R})^2 \times \mathcal{P}_*(\mathbb{R}^2) \mapsto \mathbb{R}$ is such that
  \begin{align}
 \bar{f}_{\alpha,z}(v_e, v_i , \mu)  = \sum_{a=1}^M h_a(z) \mathbb{E}^{(x,y_{e}, y_{i}) \sim \mu}\bigg[  h_a(x) f_e\big( y_{\alpha} + v_{\alpha}(x)  \big) \bigg] .
 \end{align}
 Here, for $x\in \mathcal{E}$, measurable $A \subseteq \mathcal{E}$ and measurable $B \subseteq \mathbb{R}^2$,
 \[
 \mu_{t}(A \times B) = \int_A \int_B p_{t,x}(y) dy d\kappa(x).
 \]
and $p_{t,x} \in \mathcal{C}^2\big( \mathbb{R}^2 \big)$ is such that for all $y_e,y_i \in \mathbb{R}$ and all $0 < t < \eta$
 \begin{multline}
 \partial_t p_{t,x}(y_e,y_i) = - \partial_{y_e}  \bigg( p_{t,x}(y_e,y_i) \bigg\lbrace f_e\big(y_e + \bar{v}_e(t,x) \big) - \bar{f}_{e,x}( \bar{v}_e(t) , \bar{v}_i(t) , \mu_t)  \bigg\rbrace   \bigg) \\
 - \partial_{y_i}  \bigg( p_{t,x}(y_e,y_i) \bigg\lbrace f_i\big(y_i + \bar{v}_i(t,x) \big) - \bar{f}_{i,x}( \bar{v}_e(t) , \bar{v}_i(t) , \mu_t)  \bigg\rbrace   \bigg) \\
 + \frac{1}{2}\partial_{y_e}^2\bigg( \sigma^2_e\big(x,y_e + \bar{v}_e(t,x) \big)  p_{t,x}(y_e,y_i) \bigg) +   \frac{1}{2}\partial_{y_i}^2\bigg( \sigma^2_i\big(x,y_i + \bar{v}_i(t,x) \big)  p_{t,x}(y_e,y_i) \bigg), \label{eq: Fokker Planck PDE}
 \end{multline}
 and $p_{0,x}$ is the same as in Hypothesis \ref{Hypothesis Initial Convergence}.
 \end{corollary}
\begin{proof}
   \eqref{eq: Fokker Planck PDE} is a known corollary of \eqref{eq: limiting y e x t} and \eqref{eq: limiting y i x t}. \eqref{eq: autonomous v expression} follows from applying Ito's Formula to the $mM$ equations that are necessarily satisfied for a measure in the Balanced Manifold. In more detail, the balanced requirement necessitates that for $1\leq a \leq M$ and $\alpha \in \lbrace e,i\rbrace$,
   \begin{align}
    \mathcal{G}^a_\alpha\big(\bar{v}_e(t) ,\bar{v}_i(t), \mu_t\big) = 0.
   \end{align}
Implicitly differentiating, we hence find that
   \begin{align}
 \sum_{\beta \in \lbrace e,i\rbrace} \sum_{b=1}^M \mathcal{J}^{ab}_{\alpha\beta}(\bar{v}_e(t), \bar{v}_i(t) , \mu_t) \frac{d\bar{v}^b_{\beta}}{dt}  =& -\lim_{h \to 0^+} h^{-1} \big( \mathcal{G}^a_\alpha\big(\bar{v}_e(t) ,\bar{v}_i(t) , \mu_{t+h}\big) - \mathcal{G}^a_\alpha\big(\bar{v}_e(t) ,\bar{v}_i(t)  , \mu_t) \big)\nonumber \\
 =& \mathcal{H}_{\alpha}^a( \bar{v}_e(t), \bar{v}_i(t) , \mu_t),
   \end{align}
   thanks to Ito's Lemma. We can invert $\mathcal{J}$ for all times less than $\eta$, and this yields the Corollary.
\end{proof}

\subsection{Balanced Neural Fields}
In this section we make some additional assumptions that ensure that the limiting dynamics is Gaussian (and therefore of low rank). This facilitates a `neural-field' type equation (i.e. a spatially-distributed autonomous equation for the macroscopic activity of the system). These equations would be an ideal base from which one can search for spatiotemporal patterns in balanced networks. There already exist various works that have explored pattern formation in balanced neural networks, including \cite{Litwin-Kumar2012,Rosenbaum2014,Rosenbaum2017,Ebsch2020}. Neural fields are a popular means of studying the activity in large ensembles of neurons, see the reviews in \cite{ermentrout1998neural,hutt2003pattern,Bressloff2012,Coombes2014,Cook2022}, and papers that derive them from particle limits, including \cite{Chevallier2017,Avitabile2024}. In \cite{Avitabile2024} the neural field equation was such that the solution is the local mean of a Gaussian spatially-distributed Mckean-Vlasov Fokker-Planck equation. We make analogous assumptions in this section; and this will also ensure a Gaussian limit. The Gaussian assumption ensures that the Fokker-Planck PDE for the population density \eqref{eq: Fokker Planck PDE} can be reduced to an ODE for the covariances.

In more detail, our additional assumption throughout this section is that there exists $\Sigma_e,\Sigma_i \in \mathcal{C}(\mathcal{E})$ such that
\begin{align}
f_e(z) =& - \frac{z}{\tau_e} \\
f_i(z) =& - \frac{z}{\tau_i} \\
\sigma_e(x,z) =& \Sigma_e(x) \\
\sigma_i(x,z) =& \Sigma_i(x).
\end{align}
 The Gaussian nature of the limiting probability law enables a simpler representation of the limiting dynamics. Let $\rho(v,y)$ be the Gaussian density of mean $0$, variance $v > 0$ at $y\in \mathbb{R}$, i.e.
 \begin{align}
 \rho(v,y) = \big( 2\pi v \big)^{-1/2} \exp\bigg( - \frac{y^2}{2v} \bigg).
 \end{align}
 Let $K_{e,x}(t)$ and $K_{i,x}(t)$ be the local covariances, i.e. such that
 \begin{align}
 K_{\alpha,x}(t) = \mathbb{E}^{\mu_t}\big[ y_{\alpha,x}(t)^2 \big].
 \end{align}
 and define
   \begin{align}
   \tilde{\mathcal{G}}_{\alpha}^a:& \mathcal{C}_M(\mathcal{E})^2 \times \mathcal{C}\big( \mathcal{E} , \mathbb{R}^+ \big)^2 \mapsto \mathbb{R} 
   \end{align}
 to be such that
  \begin{multline}
 \tilde{\mathcal{G}}_{\alpha}^a\big(v_e(t), v_i(t) ,  K_e(t) , K_i(t) \big) = \int_{\mathcal{E}}\int_{\mathbb{R}} h_a(z)  \bigg\lbrace  \mathcal{K}_{\alpha e}(z,x) \rho\big( K_e(x) , y \big) G_{\alpha e}(y + v_e(x) )  \\-  \rho\big( K_i(x) , y \big) \mathcal{K}_{\alpha i}(z,x)  G_{\alpha i}(y + v_i(x) )  \bigg\rbrace dy  d\kappa(x) d\kappa(z) ,
 \end{multline} 
 and we recall that for constants $\big\lbrace c_{\alpha\beta, ab} \big\rbrace $, it holds that
 for all $x,y \in \mathcal{E}$,
 \begin{align}
\mathcal{K}_{\alpha\beta}(x,y) = \sum_{a,b=1}^M c_{\alpha\beta,ab} h_a(x) h_b(y).
 \end{align}
 \begin{lemma} \label{Lemma Gaussian Neural Field}
The covariances are such that
 \begin{align}
\frac{d}{dt} K_{e,x}(t) &= - 2\tau_{e}^{-1} K_{e,x}(t) + \Sigma_{e}(x)^2 \label{eq: cov neural field 1} \\
\frac{d}{dt} K_{i,x}(t) &= - 2\tau_{i}^{-1} K_{i,x}(t) + \Sigma_{i}(x)^2 .\label{eq: cov neural field 2}
 \end{align}
 For $1\leq a \leq M$ and $\alpha \in \lbrace e,i \rbrace$, and all $t< \eta$ it must hold that
 \begin{align}
  \tilde{\mathcal{G}}_{\alpha}^a\big(v_e(t), v_i(t), K_e(t) , K_i(t) \big) = 0.
 \end{align}
 \end{lemma}
 \begin{proof}
The covariance evolutions \eqref{eq: cov neural field 1}-\eqref{eq: cov neural field 2} are standard formulae for Gaussian SDEs \cite[Section 5.6]{Karatzas1991}.
 \end{proof}
 By implicitly differentiating, we obtain the following ODE expression for the means, analogous to Corollary \ref{Corollary Autonomous Dynamics}.
 \begin{corollary}
For all $t< \eta$, the means evolve in the following manner,
 \begin{align}
 \frac{dv^a_{\alpha}}{dt} =& -\sum_{b=1}^M  \bigg( \mathcal{L}^{ab}_{\alpha e}(v(t),K_e(t), K_i(t))  \mathcal{H}_{e}^{b}(v(t), K_e(t)  ) -  \mathcal{L}^{ab}_{\alpha i}(v(t),K_e(t), K_i(t))  \mathcal{H}_{i}^{b}(v(t), K_i(t)  ) \bigg),
 \end{align}
 and the specific form of $\mathcal{L}^{ab}_{\alpha \beta}$ and $\mathcal{H}_{\alpha}^{b}$ is outlined below. 
 \end{corollary}
  Define
  \begin{align}
   \mathcal{G}_{\alpha}^a:& \mathcal{C}_M(\mathcal{E})^2 \times \mathcal{C}\big( \mathcal{E} , \mathbb{R}^+ \big)^2 \mapsto \mathbb{R} \\
 \mathcal{G}_{\alpha}^a(v, K_e, K_i ) =&   \int_{\mathcal{E}} \int_{\mathcal{E}}  h_a(x')    \bigg\lbrace \mathcal{K}_{\alpha e}(x',x)  \int_{\mathbb{R}} \rho\big(  K_{e,t}(x) , y_e \big) G_{\alpha e}(y_{e} + v_e(x) ) dy_e \nonumber \\ &-  \mathcal{K}_{\alpha i}(x',x)   \int_{\mathbb{R}} \rho\big(   K_{i,t}(x) , y_i \big) G_{\alpha i}(y_{i} + v_i(x) ) dy_i \bigg\rbrace d\kappa(x)  d\kappa(x') .
 \end{align}
 Analogously to the definitions in \eqref{eq: Jacobian e} and \eqref{eq: Jacobian i}, for $1\leq a,b \leq M$ and $\alpha\in \lbrace e,i \rbrace$, define   
   \begin{align}
   \mathcal{J}_{\alpha e}^{ab}:& \mathcal{C}_M(\mathcal{E})^2 \times \mathcal{C}\big( \mathcal{E} , \mathbb{R}^+ \big)^2 \mapsto \mathbb{R} \\
 \mathcal{J}_{\alpha e}^{ab}(v, K_e, K_i ) =&   \int_{\mathcal{E}} \int_{\mathcal{E}}  h_a(x')     \mathcal{K}_{\alpha e}(x',x)  \int_{\mathbb{R}} \rho\big(   K_{e}(x) , y_e \big)h_b(x) \dot{G}_{\alpha e}(y_{e} + v_e(x) ) dy_e d\kappa(x)  d\kappa(x') \\
 \mathcal{J}_{\alpha i}^{ab}(v, K_e, K_i ) =&  -  \int_{\mathcal{E}} \int_{\mathcal{E}}  h_a(x')    \mathcal{K}_{\alpha i}(x',x)   \int_{\mathbb{R}} \rho\big(  K_{i}(x) , y_i \big) h_b(x) \dot{G}_{\alpha i}(y_{i} + v_i(x) ) dy_i  d\kappa(x)  d\kappa(x') ,
 \end{align}
 and let $\mathcal{J}:  \mathcal{C}_M(\mathcal{E})^2 \times \mathcal{C}\big( \mathcal{E} , \mathbb{R}^+ \big)^2 \mapsto \mathbb{R}^{2M\times 2M}$ be the matrix with elements $\big( \mathcal{J}_{\alpha e}^{ab} \big)_{a,b \leq M \fatsemi \alpha,\beta \in \lbrace e,i \rbrace}$. Write $\mathcal{L}:   \mathcal{C}_M(\mathcal{E})^2 \times \mathcal{C}\big( \mathcal{E} , \mathbb{R}^+ \big)^2 \mapsto \mathbb{R}^{2M\times 2M}$ be such that $\mathcal{L}(v,K_e,K_i) := \mathcal{J}(v,K_e,K_i)^{-1}$. Finally, analogously to \eqref{eq: autonomous v expression}, define
     \begin{align}
   \mathcal{H}_{\alpha }^{a}:& \mathcal{C}_M(\mathcal{E})^2 \times \mathcal{C}\big( \mathcal{E} , \mathbb{R}^+ \big)^2 \mapsto \mathbb{R} \\
 \mathcal{H}_{e}^{a}(v, K_e  ) =&   \int_{\mathcal{E}} \int_{\mathcal{E}}  h_a(x')     \mathcal{K}_{\alpha e}(x',x)  \int_{\mathbb{R}}\bigg( -\frac{1}{2 K_e(x)} + \frac{y_e^2}{2K_e(x)^2} \bigg) \nonumber \\ &\times \dot{K}_e(x)  \rho\big(   K_{e}(x) , y_e \big) G_{\alpha e}(y_{e} + v_e(x) ) dy_e d\kappa(x)  d\kappa(x') \text{ where } \\
\dot{K}_e(x) =& - 2\tau_{e}^{-1} K_{e}(x) + \Sigma_{e}(x)^2 \\
 \mathcal{H}_{i}^{a}(v, K_i  ) =&   \int_{\mathcal{E}} \int_{\mathcal{E}}  h_a(x')     \mathcal{K}_{\alpha i}(x',x)  \int_{\mathbb{R}}\bigg( -\frac{1}{2 K_i(x)} + \frac{y_i^2}{2K_i(x)^2} \bigg)\nonumber \\ &\times \dot{K}_i(x)\rho\big(   K_{i}(x) , y_i \big) G_{\alpha i}(y_{i} + v_i(x) ) dy_i d\kappa(x)  d\kappa(x') \\
 \dot{K}_i(x) =& - 2\tau_{i}^{-1} K_{i}(x) + \Sigma_{i}(x)^2.
 \end{align}
  We now outline the proof of Lemma \ref{Lemma Gaussian Neural Field}.
\begin{proof}
We must show that the dynamics is consistent with that in Corollary \ref{Corollary Autonomous Dynamics}, i.e. we must show that 
\begin{align}
    \mathcal{H}^a_{\alpha}(v,K_{\alpha,t}) = \mathcal{H}^a_{\alpha}(v,\mu_t) .\label{eq: to show H equivalence}
\end{align}
It is immediate from the definition of the Gaussian density that
\begin{align}
\frac{\partial}{\partial t} \rho\big( K_{\alpha,t}(x) , y \big) =  \bigg( - \frac{1}{2 K_{\alpha}(x)} + \frac{y_{\alpha}^2}{2 K_{\alpha}(x)^2} \bigg)\rho\big( K_{\alpha,t}(x) , y \big)\frac{\partial}{\partial t} K_{\alpha,t}(x),
\end{align}
and \eqref{eq: to show H equivalence} readily follows.
 \end{proof}
 
 \section{Numerical Investigation}
 
For our numerical investigations, we start with the mean-field scenario, with no spatial extension. In this case, $\mathcal{E} = 0$,
$M=1$ and $h_1(0) = 1$. We can also take $\sigma_e$ and $\sigma_i$ to both be equal to $1$, and $f_e(z) = - z / \tau_e$ and $f_i(z) = -z / \tau_i$. We thus find that the $n$-dimensional particle model solves the system of SDEs
\begin{align}
dz^j_{e,t}  =& \bigg\lbrace  - z^j_{e,t} / \tau_e + n^{-1/2} \sum_{k\in I_n}\big( G_{ee}(z^k_{e,t}) - G_{ei}(z^k_{i,t}) \big) \bigg\rbrace dt  +  \sigma_e dW^j_{e,t} \\
dz^j_{i,t} =& \bigg\lbrace -z^j_{i,t} / \tau_i + n^{-1/2} \sum_{k\in I_n}\big(  G_{ie}(z^k_{e,t}) - G_{ii}(z^k_{i,t}) \big) \bigg\rbrace dt  + \sigma_i dW^j_{i,t} .
\end{align}

\subsection{First Test: Inhibition-Stabilized Networks}
 
\begin{figure}[htbp]
  \centering
  \includegraphics[width=0.75\textwidth]{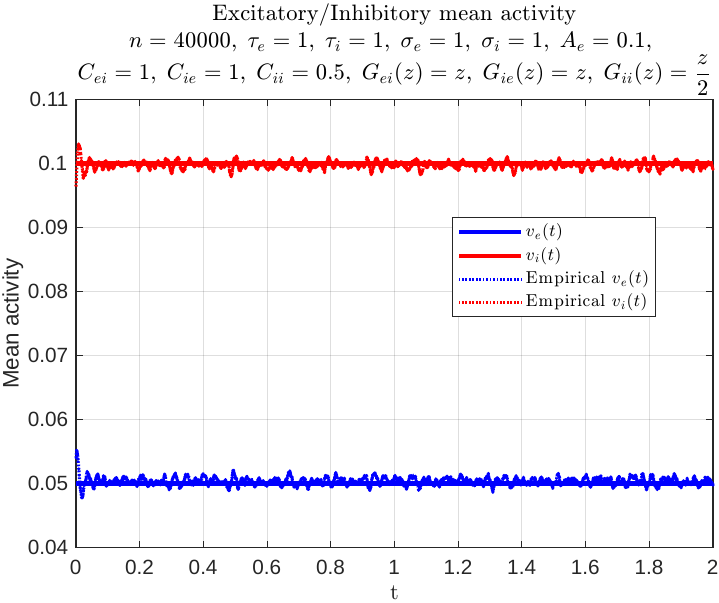}
  \caption{%
    Time evolution of the excitatory (blue) and inhibitory (red) mean population activities.
    Solid lines show the deterministic predictions \(v_e(t)\) and \(v_i(t)\) from \eqref{eq:ODEvs41},
    while the dotted lines plot empirical averages from the stochastic simulation. The parameters are $n=40000$, $\tau_e = 1$, $\tau_i = 1$, $\sigma_e=1$, $\sigma_i=1$, $A_e=1$, $C_{ei} = 1$, $C_{ie} = 1$, $C_{ii} = 0.5$, $G_{ei}(z) = z$, $G_{ie}(z) = z$, $G_{ii}(z) = z/2$.
  }
  \label{fig:mean_activity}
\end{figure}

\begin{figure}[htbp]
  \centering
  \includegraphics[width=0.75\textwidth]{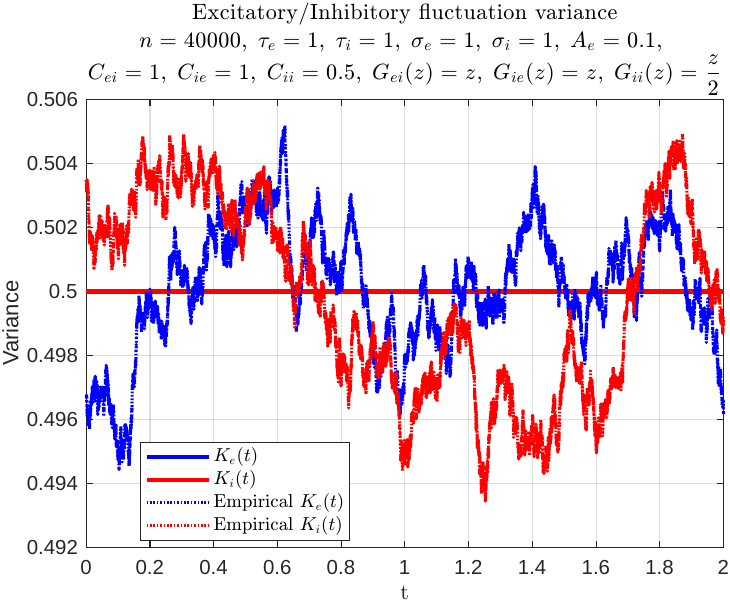}
  \caption{%
    Fluctuation variances of the excitatory (\(K_e(t)\), blue) and inhibitory (\(K_i(t)\), red) populations.
    Solid curves are the analytic solutions of \eqref{eq:ODEKs41} started at steady‐state
    \(\tau\,\sigma^2/2\), and dotted curves show the sample variances from the stochastic simulation. The parameters are the same as in Figure \ref{fig:mean_activity}.%
  }
  \label{fig:variance_activity}
\end{figure}
We start with linear inhibition-stabilized networks (a popular model in neuroscience, see for instance \cite{tsodyks1997paradoxical,ferguson2020mechanisms}). In this case, excitation is \textit{constant}, so we take
\begin{align}
G_{ee}(z) =& A_{e} > 0 \\
G_{ie}(z_e) =& C_{ie}z_e \\
G_{ii}(z_i) =& C_{ii}z_i \\
G_{ei}(z_i) =& C_{ei}z_i ,
\end{align}
and every constant is positive. One obtains that, in the limit $\mu_t$ is centered and Gaussian, with variances $(K_e(t), K_i(t))$ satisfying the ODE
\begin{align}
\label{eq:ODEKs41}
\frac{d}{dt} K_e(t) &= -\frac{2}{\tau_e} K_e(t) + \sigma_e^2 \\
\frac{d}{dt} K_i(t) &= -\frac{2}{\tau_i} K_i(t) + \sigma_i^2.
\end{align}
The limiting means are such that
\begin{align}
\label{eq:ODEvs41}
v_i(t) =& A / C_{ei} \\
v_e(t) =& \frac{C_{ii}}{C_{ie}} v_i(t)
\end{align}
A comparison of the mean and variance in the deterministic and stochastic systems is plotted in Figures \ref{fig:mean_activity} and \ref{fig:variance_activity}. 

\subsection{Second Test: Nonlinear Interactions}

Now lets take some nonlinear interactions, but keep the intrinsic dynamics linear. This ensures a Gaussian limit. We still consider the inhibition-stabilized scenario. As previously, the variances are such that
\begin{align}
\frac{d}{dt} K_e(t) &= -\frac{2}{\tau_e} K_e(t) + \sigma_e^2 \\
\frac{d}{dt} K_i(t) &= -\frac{2}{\tau_i} K_i(t) + \sigma_i^2.
\end{align}
Write 
\begin{align}
G_{ee}(z) =& A_{e} > 0.
\end{align}
We take $G_{ie}, G_{ii}, G_{e,i}$ to be increasing sigmoidal functions (i.e. bounded and smooth): in the numerics below they are of the form
\[
G_{\alpha\beta}(x) = C_{\alpha\beta}\tanh\big( \gamma_{\alpha\beta}(x-\xi_{\alpha\beta}) \big).
\]
We require that
\begin{align}
 C_{ei} > A_e,
\end{align}
which is necessary to ensure that the inhibition is able to control the excitation. 
We write the Gaussian kernel to be
\begin{align}
\rho(m,V,y) = \big(2\pi V \big)^{-1/2} \exp\bigg( - \frac{(y-m)^2}{2V} \bigg).
\end{align}
Notice that
\begin{align}
\partial_m \rho(m,V,y) =& \frac{y-m}{V}\rho(m,V,y) \\
\partial_V \rho(m,V,y) =& \rho(m,V,y) \bigg( -\frac{1}{2V} + \frac{(y-m)^2}{2V^2} \bigg)
\end{align}
We require that for all time,
\begin{align}
 \int_{\mathbb{R}} G_{ei}(y) \rho\big(v_i,K_i,y\big) dy &= A_e \\
  \int_{\mathbb{R}} G_{ii}(y) \rho\big(v_i,K_i,y\big) dy &=  \int_{\mathbb{R}} G_{ie}(y) \rho\big(v_e,K_e,y\big) dy.
\end{align}
Differentiating the above equations with respect to time, we obtain that 
\begin{align}
\label{eq:ODEvs42}
\frac{dv_i}{dt} = - \bigg( \int_{\mathbb{R}} G_{ei}(y) \frac{y-v_i}{K_i(t)} \rho(v_i,K_i(t),y) dy \bigg)^{-1} \bigg(  -\frac{2}{\tau_i} K_i(t) + \sigma_i^2 \bigg) \times \\
\int_{\mathbb{R}}G_{ei}(y)\rho\big(v_i(t),K_i(t),y\big) \bigg( -\frac{1}{2K_i(t)} + \frac{(y-v_i(t))^2}{2K_i(t)^2} \bigg) dy
\end{align}
and $\frac{dv_e}{dt}$ is such that
\begin{multline}
\frac{dv_e}{dt}  \int_{\mathbb{R}} G_{ie}(y) \frac{y-v_e(t)}{K_e(t)} \rho(v_e(t) ,K_e(t),y) dy  - \frac{dv_i}{dt}  \int_{\mathbb{R}} G_{ii}(y) \frac{y-v_i(t)}{K_i(t)} \rho(v_i(t),K_i(t),y) dy + \nonumber \\  \bigg(  -\frac{2}{\tau_e} K_e(t) + \sigma_e^2 \bigg)  \int_{\mathbb{R}} G_{ie}(y)  \bigg( -\frac{1}{2K_e(t)} + \frac{(y-v_e(t))^2}{2K_e(t)^2} \bigg)\rho(v_e(t),K_e(t),y) dy \\
- \bigg(  -\frac{2}{\tau_i} K_i(t) + \sigma_i^2 \bigg)  \int_{\mathbb{R}} G_{ii}(y)  \bigg( -\frac{1}{2K_i(t)} + \frac{(y-v_i(t))^2}{2K_i(t)^2} \bigg)\rho(v_i(t) ,K_{i}(t),y) dy =0.
\end{multline}
For the initial conditions, we take $K_e(0)$ and $K_i(0)$ to be arbitrary positive constants. The initial conditions must enforce balanced excitation / inhibition,
and we therefore choose $v_i(0)$ to be such that 
\begin{align}
\int_{\mathbb{R}}G_{ei}(y) \rho\big(v_i(0) , K_i(0), y\big) dy = A_e. \label{eq: v i 0 initial second test}
\end{align}
Once we have solved for $v_i(0)$, we let $v_e(0)$ be such that
\begin{align}
\int_{\mathbb{R}} G_{ii}(y) \rho\big(v_i(0),K_i(0),y\big) dy =  \int_{\mathbb{R}} G_{ie}(y) \rho\big(v_e(0),K_e(0),y\big) dy.\label{eq: v e 0 initial second test}
\end{align}
In practice, we solve \eqref{eq: v i 0 initial second test} and \eqref{eq: v e 0 initial second test} using MATLAB's nonlinear solver.

Finally, we must specify the $n$-dimensional particle system that converges to the above limit. This is obtained by choosing the initial conditions to be such that
\begin{align}
z^j_{e,0} &= v_e(0) + \sqrt{K_e(0)} \tilde{z}^j_{e,0} \\
z^j_{i,0} &= v_i(0) +\sqrt{K_i(0)} \tilde{z}^j_{i,0} ,
\end{align}
where $\lbrace \tilde{z}^j_{e,0} , \tilde{z}^j_{i,0} \rbrace_{j\in I_n}$ are iid $\mathcal{N}(0,1)$. A comparison of the stochastic and deterministic systems is plotted in Figures \ref{fig:mean_activity42} and \ref{fig:variance_activity42}. 

\begin{figure}[htbp]
  \centering
  \includegraphics[width=0.75\textwidth]{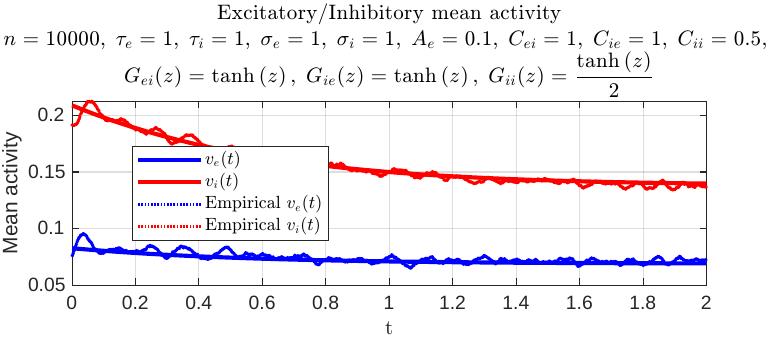}
  \caption{%
    Dynamics of the excitatory (blue) and inhibitory (red) mean activities under the nonlinear gain \(G_{\alpha\beta}(z)=C_{\alpha\beta}\tanh(z)\).
    Solid curves show the deterministic predictions \(v_e(t)\) and \(v_i(t)\) obtained by solving the four‐variable ODEs from \eqref{eq:ODEvs42},
    while dotted curves trace the empirical averages from the stochastic simulation
    (\(n=10\,000\), \(\tau_e=\tau_i=1\), \(\sigma_e=\sigma_i=1\), \(A_e=0.1\), \(C_{ei}=C_{ie}=1\), \(C_{ii}=0.5\)).%
  }
  \label{fig:mean_activity42}
\end{figure}

\begin{figure}[htbp]
  \centering
  \includegraphics[width=0.75\textwidth]{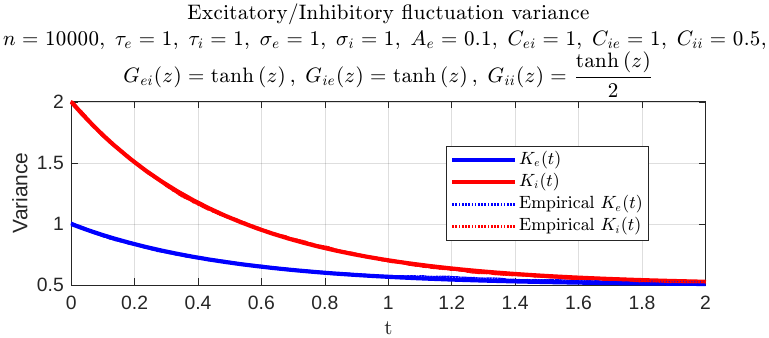}
  \caption{%
    Evolution of the fluctuation variances \(K_e(t)\) (blue) and \(K_i(t)\) (red) under the same nonlinear‐gain setup.
    Solid lines are the analytic variance ODE solutions of \eqref{eq:ODEKs41} with initial conditions \(K_{e}(0)=1\), \(K_{i}(0)=2\),
    and dotted lines are the sample variances from the stochastic simulation. The parameters are the same as Figure \ref{fig:mean_activity42}.%
  }
  \label{fig:variance_activity42}
\end{figure}

 \subsection{Third Test: Spatially-Distributed Connectivity Profile}
For our third test, we take $\mathcal{E} = \mathbb{S}^1 := (-\pi , \pi]$. Take $h_1(\theta) = 1$, $h_2(\theta) = \cos(\theta)$. Write 
\begin{align}
\mathcal{K}_{\alpha\beta}( x, x') = \sum_{i=1}^3 c_{\alpha\beta,ii}  h_i(x) h_i(x'),
\end{align}
where the constants are such that  $c_{\alpha\beta,22} = c_{\alpha\beta,33}$. Neuron $j$ is assigned the position $x^j_n = 2\pi j / n$. We take $G_{\alpha\beta}(z)$ to be sigmoidal, focussing on the hyperbolic tangent
\begin{align}
G_{\alpha\beta}(z) = C_{\alpha\beta} \tanh( \gamma_{\alpha\beta} z ).
\end{align}
As previously, the limiting variance dynamics is 
\begin{align}
\frac{d}{dt} K_e(t) &= -\frac{2}{\tau_e} K_e(t) + \sigma_e^2 \\
\frac{d}{dt} K_i(t) &= -\frac{2}{\tau_i} K_i(t) + \sigma_i^2.
\end{align}
For convenience we take the initial variances to be their equilibrium values
\begin{align}
K_e(0) &= K_e^* = \frac{\tau_e \sigma_e^2}{2} \\
K_i(0) &= K_i^* =  \frac{\tau_i \sigma_i^2}{2} .
\end{align}
The means are such that
\begin{align}
m_e(t,\theta) =& v^1_e(t) + v^2_e(t) \cos(\theta) \\ 
m_i(t,\theta) =& v^1_i(t) + v^2_i(t) \cos(\theta).
\end{align}
\begin{figure}[htbp]
  \centering
  \includegraphics[width=0.45\textwidth]{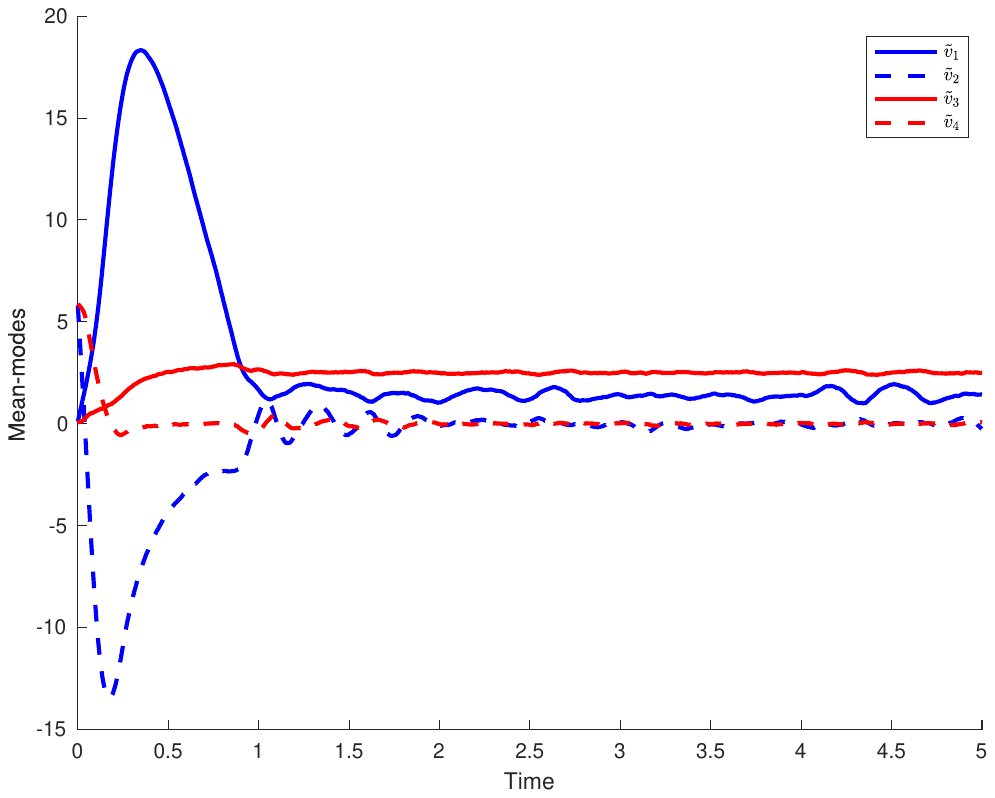}
    \includegraphics[width=0.45\textwidth]{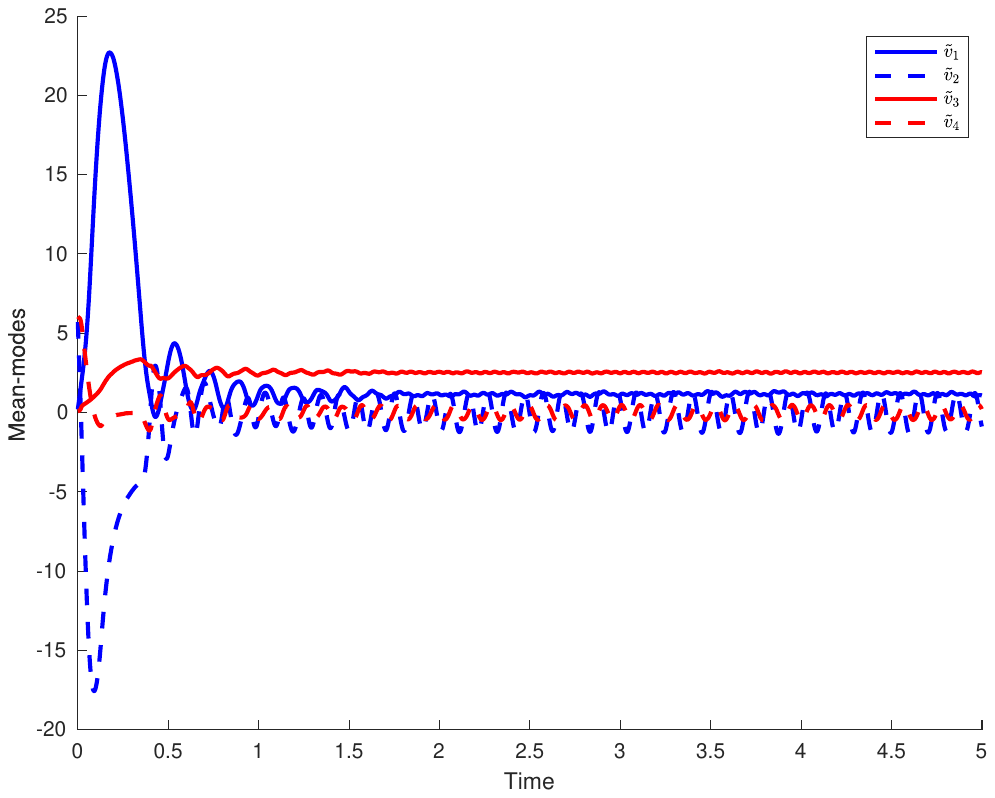}
  \caption{%
Stochastic evolution of the average excitation and inhibition. On the left $n=100$, and on the right $n=500$. The other parameters are $\tau_e = \tau_i = 0.5$, $\sigma_e = \sigma_i = 0.5$, $c_{ee} = [0.5,2]$, $c_{ei} = [4,4]$, $c_{ie} = [1,2]$ and $c_{ii} = [1,2]$. There does not seem to exist a balanced state in these simulations. As $n\to\infty$, oscillations arise, and the frequency increases with $n$.
  }
  \label{fig:variance_activity spatial}
\end{figure}

We require that the initial means are such that 
  \begin{align}
  \tilde{\mathcal{G}}^a(v(t), K_e(t), K_i(t) ) =& 0 \text{ where }\\
   \tilde{\mathcal{G}}^a:& \mathbb{R}^4 \times \mathbb{R} \times \mathbb{R} \mapsto \mathbb{R} \\
\tilde{\mathcal{G}}^1(v, K_e, K_i ) =&   \int_{\mathcal{E}}  \int_{\mathbb{R}} \bigg\lbrace \rho\big(  K_{e} , y \big) G_{e e}(y + m_e( \theta) )  -    \rho\big(   K_{i} , y \big) G_{e i}(y + m_i(\theta) ) \bigg\rbrace  dy d\theta \\
\tilde{\mathcal{G}}^2(v, K_e, K_i ) =&   \int_{\mathcal{E}}  \int_{\mathbb{R}} \bigg\lbrace \rho\big(  K_{e} , y \big) G_{i e}(y + m_e( \theta) )  -    \rho\big(   K_{i} , y \big) G_{i i}(y + m_i(\theta) ) \bigg\rbrace  dy d\theta \\
\tilde{\mathcal{G}}^3(v, K_e, K_i ) =&   \int_{\mathcal{E}} \cos(\theta) \int_{\mathbb{R}} \bigg\lbrace \rho\big(  K_{e} , y \big) G_{e e}(y + m_e( \theta) )  -    \rho\big(   K_{i} , y \big) G_{e i}(y + m_i(\theta) ) \bigg\rbrace  dy d\theta \\
\tilde{\mathcal{G}}^4(v, K_e, K_i ) =&   \int_{\mathcal{E}} \cos(\theta)  \int_{\mathbb{R}} \bigg\lbrace \rho\big(  K_{e} , y \big) G_{i e}(y + m_e( \theta) )  -    \rho\big(   K_{i} , y \big) G_{i i}(y + m_i(\theta) ) \bigg\rbrace  dy d\theta 
 \end{align}
For the stochastic simulations, the state of the $j^{th}$ excitatory neuron is written as $z^j_{e,t} \in \mathbb{R}$, and the state of the $j^{th}$ inhibitory neuron is written as $z^j_{i,t} \in \mathbb{R}$. The dynamics is assumed to be of the form
\begin{align}
dz^j_{e,t}  =& \bigg\lbrace  -z^j_{e,t} / \tau_e + n^{-1/2} \sum_{k\in I_n}\bigg(  \mathcal{K}_{ee}(x^j_n , x^k_n) G_{ee}(z^k_{e,t}) - \mathcal{K}_{ei}(x^j_n , x^k_n)  G_{ei}(z^k_{i,t}) \bigg) \bigg\rbrace dt \nonumber \\ &+ \sigma_e dW^j_{e,t} \\
dz^j_{i,t} =& \bigg\lbrace  -z^j_{i,t} / \tau_i + n^{-1/2} \sum_{k\in I_n}\bigg( \mathcal{K}_{ie}(x^j_n , x^k_n) G_{ie}(z^k_{e,t}) -\mathcal{K}_{ii}(x^j_n , x^k_n)  G_{ii}(z^k_{i,t}) \bigg) \bigg\rbrace dt  \nonumber \\ &+ \sigma_i dW^j_{i,t} .
\end{align}
Here $\lbrace W^j_{e,t} , W^j_{i,t} \rbrace_{j\in I_n}$ are independent Brownian Motions. The initial conditions $\lbrace z^j_{e,0} , z^j_{i,0} \rbrace_{j\in I_n}$ are such that
\begin{align}
z^j_{e,0} &= m_e(0,x^j_n) + \zeta^j_e \sqrt{K_e(0)} \\
z^j_{i,0} &= m_i(0,x^j_n) + \zeta^j_i \sqrt{K_i(0)}.
\end{align}
where $\lbrace \zeta^j_e , \zeta^j_i \rbrace_{j\in I_n}$ are independent $\mathcal{N}(0,1)$. 

For this simple system, we were not able to find non-trivial spatially-extended balanced solutions. However it seems likely that spatially-extended balanced solutions might exist for a more complicated interaction structure (i.e. a `Mexican Hat' connectivity). This will be explored in a future work. In Figure \ref{fig:variance_activity spatial}, we plot a numerical simulation of the particle system, for which the system does not converge to a balanced state. However it intriguingly seems to exhibit oscillatory activity, with the frequency of the oscillations increasing with the system size $n$. This will be the subject of future research.

 
 \section{Discussion}
 We have determined autonomous equations that describe the dynamics of balanced neural networks. Previous treatments have determined equations for the existence of a balanced state by taking the Fourier Transform in time, and requiring that this state is balanced. We instead prove that one only requires the balanced manifold to be linearly stable to perturbations in the same direction as the mean-field interaction.
 
It is immediate from our theorem that the high-dimensional balanced neural network will exhibit similar characteristics to that described in much of the literature \cite{VanVreeswijk1996,VanVreeswijk2003}. The neurons will be mostly decorrelated, asynchronous and temporally irregular. However in directions that lie in the span of the basis for the interactions, (i.e. $\lbrace h_a(x) \rbrace_{a=1}^M$), there will be very little variability (for large $n$), as has been predicted in (for instance) \cite{Rosenbaum2014,Darshan2018}.

An important question is to ask for which choices of parameters does the balanced manifold exist?  More particularly, when will the system stay near the balanced manifold for long periods of time? There already exists a literature that addresses this question. It is widely thought that the balanced manifold exists whenever there is fast inhibition; i.e. such that any transient increase in the activity of excitatory neurons almost immediately triggers an increase in inhibitory activity, which will counter the system blowing up.
 
 \section{Proofs}
 
 We begin by providing an overview of the method of proof. Specific details are provided in Section \ref{Subsection Proof Details}, including the proof of the existence of the limiting distribution in Lemma \ref{Lemma Existence of Hydro Limit}.
 
 Fix some $0 < T < \eta$. It follows from the definition of the balanced manifold that there exists $\xi_T > 0$ such that for all $t\leq T$, every eigenvalue of $\mathcal{J}\big( \bar{v}_e(t) ,\bar{v}_i(t) ,\mu_t\big)$ has real component less than or equal to $-\xi_T$. 
 
Lets start by determining a specific SDE for $\lbrace y^j_{e,t} , y^j_{i,t} \rbrace$. To this end, writing
  \begin{align}
  d\bar{w}^a_{e,t} =& n^{-1}\sum_{j\in I_n} \sum_{b=1}^M \mathcal{Q}^{-1}_{ab} h_b(x^j_n) \sigma_e(x^j_n , z^j_{e,t}) dW^j_{e,t} \\
   d\bar{w}^a_{i,t} =& n^{-1}\sum_{j\in I_n} \sum_{b=1}^M  \mathcal{Q}^{-1}_{ab} h_b(x^j_n) \sigma_i(x^j_n , z^j_{i,t}) dW^j_{i,t} 
  \end{align}
it follows from Ito's Lemma that
\begin{align}
 dv^a_{e,t}  = & n^{-1} \sum_{j\in I_n} \sum_{b=1}^M \mathcal{Q}^{-1}_{ab} h_b(x^j_n)\bigg\lbrace f_e(z^j_{e,t}) \nonumber \\ &+ n^{-1/2} \sum_{k\in I_n} \bigg(  \mathcal{K}_{ee}(x^j_n , x^k_n) G_{ee}(z^k_{e,t}) - \mathcal{K}_{ei}(x^j_n , x^k_n)  G_{ei}(z^k_{i,t}) \bigg) \bigg\rbrace dt  
  + d\bar{w}^a_{e,t} \label{eq: dv et} \\
 dv^a_{i,t}  =& n^{-1} \sum_{j\in I_n}\sum_{b=1}^M \mathcal{Q}^{-1}_{ab} h_b(x^j_n)  \bigg\lbrace f_i(z^j_{i,t}) \nonumber \\ &+ n^{-1/2} \sum_{k\in I_n} \bigg(  \mathcal{K}_{ie}(x^j_n , x^k_n) G_{ie}(z^k_{e,t}) - \mathcal{K}_{ii}(x^j_n , x^k_n)  G_{ii}(z^k_{i,t}) \bigg) \bigg\rbrace dt 
 + d\bar{w}^a_{i,t}, \label{eq: dv it}
\end{align}  
and
\begin{align}
dy^j_{e,t} =& \bigg( f(z^j_{e,t}) - n^{-1} \sum_{k\in I_n} \sum_{a,b = 1}^{M} Q^{-1}_{ab} f(z^k_{e,t}) h_a(x^k_{n})h_b(x^j_{n})   \bigg) dt \nonumber \\& + \sigma(x^j_n , z^j_{e,t})dW^j_{e,t} - \sum_{a=1}^M h_a(x^j_n)  d\bar{w}^a_{e,t} \label{eq: y j e t} \\
dy^j_{i,t} =& \bigg( f(z^j_{i,t}) -  n^{-1} \sum_{k\in I_n}  \sum_{a,b = 1}^{M} Q^{-1}_{ab}  f(z^k_{i,t}) h_a(x^k_{n}) h_b(x^j_{n})   \bigg) dt \nonumber \\ &  + \sigma(x^j_n, z^j_{i,t})dW^j_{i,t} - \sum_{a=1}^M h_a(x^j_n) d\bar{w}^a_{i,t} .\label{eq: y j i t} 
\end{align}
The initial values are such that
  \begin{align}
  y^j_{e,0} &= z^j_{e,0} -  \sum_{a = 1}^M v^a_e(0) h_a(x^j_n) \\
    y^j_{i,0} &= z^j_{i,0} -  \sum_{a = 1}^M v^a_i(0) h_a(x^j_n).
  \end{align}
Write
 \begin{align}
\ell_s =&   \norm{v_{e}(s) - \bar{v}_{e}(s) }^2 + \norm{v_{i}(s) - \bar{v}_i(s) }^2 \nonumber \\
=& \sum_{a=1}^M\bigg( \big( v^a_{e}(s) - \bar{v}^a_{e}(s) \big)^2 + \big( v^a_{i}(s) - \bar{v}^a_{i}(s) \big)^2 \bigg).
 \end{align}
 Our first main result is to bound the growth of $\ell_s$. The implication of the following Lemma is that, in a neighborhood of the balanced manifold, the growth of $\ell_s$ is dominated by the attractive pull of the balanced manifold.
 \begin{lemma} \label{Lemma Bound the frowth of ell t}
 There exist constants $\epsilon_T , c_T > 0$ (independent of $n$) such that for any stopping times $\zeta  < \iota \leq T $, as long as 
 \begin{align}
\sup_{r\leq \iota} d_W\big( \hat{\mu}^n_r , \mu_r \big) &\leq \epsilon_T \text{ and } \label{eq: ell bound cond 1}\\
\sup_{r\leq \iota} \ell_r &\leq \epsilon_T
\end{align}
then necessarily, it holds that 
\begin{multline}
\ell(\iota)  \leq \ell(\zeta)+ \tilde{w}_\iota -\tilde{w}_\zeta 
+ \int_\zeta^\iota \bigg\lbrace \sqrt{\ell(s)} c_T + c_T n^{-1}  + \sqrt{n}\bigg( c_T\sqrt{\ell(s)} d_W\big( \hat{\mu}^n_s , \mu_s \big)  -  \xi_T  \ell(s) / 2 \bigg)  \bigg\rbrace ds
\end{multline}
and $\tilde{w}_t$ is a continuous Martingale adapted to $\mathcal{F}$ with quadratic variation upperbounded by \newline $n^{-1}  c_T \int_\zeta^\iota \ell(s) ds $.
 \end{lemma}
 
\begin{corollary} \label{Corollary Bound ell t}
\begin{multline}
\lsup{n} n^{-1/3} \log \mathbb{P}\bigg( \text{ For some }t\leq T, \text{ it holds that }\sqrt{\ell_t} >\big( n^{-1/4} +  2 \frac{c_T}{\xi_T} d_W\big( \hat{\mu}^n_t , \mu_t \big)\big) \\ \text{ and } \sup_{s\leq t} d_W\big( \hat{\mu}^n_s , \mu_s \big) \leq \epsilon_T \text{ and }\sup_{s \leq t}   \ell_s   \leq \epsilon_T \bigg) < 0.
\end{multline}
\end{corollary}
We next define an approximate set of SDEs $\big\lbrace \bar{y}^j_{\alpha,t} \big\rbrace_{\alpha\in\lbrace e,i \rbrace, j \in I_n}$. The definition of $\bar{y}^j_{\alpha,t}$ is the same as the definition of  $y^j_{\alpha,t}$, except that $v_t$ is replaced by $\bar{v}_t$ ($\bar{v}_t$ is the limiting mean that exists thanks to Lemma \ref{Lemma Existence of Hydro Limit}). They are defined to be the strong solution of the system of SDEs 
\begin{align}
d\bar{y}^{j}_{e,t} =& \bigg( f_e(\bar{z}^{j}_{e,t}) - n^{-1} \sum_{a,b = 1}^{M} \sum_{k\in I_n} Q^{-1}_{ab}  f_e(\bar{z}^k_{e,t})h_b(x^k_{n})  h_a(x^j_{n})  \bigg) dt  + \sigma_e(x^j_n , \bar{z}^{j}_{e,t})dW^j_{e,t}  \label{eq: y j e t approximate} \\
d\bar{y}^{j}_{i,t} =& \bigg( f_i(\bar{z}^{j}_{i,t}) - n^{-1} \sum_{a,b = 1}^{M} \sum_{k\in I_n} Q^{-1}_{ab}  f_i(\bar{z}^k_{i,t})h_b(x^k_{n})  h_a(x^j_{n})  \bigg) dt  + \sigma_i(x^j_n , \bar{z}^{j}_{i,t})dW^j_{i,t}  ,\label{eq: y j i t approximate} 
\end{align}
where
\begin{align}
\bar{z}^{j}_{e,t} &= \bar{y}^{j}_{e,t} + \sum_{a = 1}^M h_a\big( x^j_n \big) \bar{v}^{a}_{e}(t) \\
\bar{z}^{j}_{i,t} &= \bar{y}^{j}_{i,t} + \sum_{a = 1}^M  h_a\big( x^j_n \big) \bar{v}^{a}_{i}(t)
\end{align}
with initial conditions $\bar{y}^{j}_{e,0} = y^j_{e,0}$ and  $\bar{y}^{j}_{i,0} = y^j_{i,0}$. Next, we note the existence of a strong solution.
\begin{lemma} \label{Lemma Strong Intermediate}
There exists a unique strong solution to \eqref{eq: y j e t approximate} - \eqref{eq: y j i t approximate}.
\end{lemma}
The proof of Lemma \ref{Lemma Strong Intermediate} is a standard result in SDE Theory \cite{Karatzas1991} and is omitted. Let $\tilde{\mu}^n_t \in \mathcal{P}\big( \mathcal{E} \times \mathbb{R}^2 \big)$ be the empirical measure at time $t$, i.e.

\begin{align}
\tilde{\mu}^n_t = n^{-1}\sum_{j\in I_n} \delta_{x^j_n , \bar{y}^j_{e,t} , \bar{y}^j_{i,t}}.
\end{align}
Next we notice that $\tilde{\mu}^n_t$ must converge to $\mu_t$ as $n\to\infty$.
\begin{lemma} \label{Lemma Law of Large Numbers}
$\mathbb{P}$-almost-surely,
\begin{align}
\lim_{n\to\infty} \sup_{t\leq T} d_W\big( \tilde{\mu}^n_t , \mu_t \big) = 0.
\end{align}
\end{lemma}
\begin{proof}
$\big\lbrace \bar{y}^{j}_{e,[0,t]} , \bar{y}_{i,[0,t]}^j \big\rbrace_{j\in I_n}$ are independent  $\mathcal{C}\big( [0,T] , \mathbb{R} \big)^2$-valued random variables. The Lemma therefore follows from the Law of Large Numbers.  
\end{proof}
Write $y^{j}_{\alpha,t} - \bar{y}^j_{\alpha,t} = u^j_{\alpha,t}$. We find that
\begin{multline*}
du^j_{\alpha,t} = \bigg(f_\alpha(z^{j}_{\alpha,t}) - f_\alpha(\bar{z}^{j}_{\alpha,t}) - n^{-1} \sum_{k\in I_n} \sum_{a,b = 1}^{M} Q^{-1}_{ab} h_a(x^j_{n})h_b(x^k_{n})\bigg( f_\alpha(z^{k}_{\alpha,t}) - f_\alpha(\bar{z}^{k}_{\alpha,t})  \bigg)  \bigg)  dt \\ + \bigg(\sigma_\alpha(x^j_n, z^{j}_{\alpha,t}) -  \sigma_\alpha(x^j_n, \bar{z}^{j}_{\alpha,t}) \bigg) dW^j_{\alpha,t}
\end{multline*}
Define the stopping times 
\begin{align}
\zeta_n =& \inf\bigg\lbrace t \leq T  \; : \;   \sqrt{\ell_t} =  n^{-1/4} +   \frac{2 c_T}{\xi_T} d_W\big( \hat{\mu}^n_t , \mu_t \big)      \bigg\rbrace  .
\end{align}
Next write
\[
q_t = n^{-1}  \sum_{j\in I_n}\big( (u^j_{e,t})^2 + (u^j_{i,t})^2 \big).
\]
\begin{lemma}\label{Lemma Bound q t n quarter}
There exists a constant $C > 0$ (independent of $n$) such that
\begin{align}
\lsup{n} n^{-1/2} \log \mathbb{P}\bigg( \text{ There exists }t\leq \zeta_n \text{ such that }q_t \geq C n^{-1/4} \bigg) < 0.
\end{align}
\end{lemma}
%
%
%
We can now prove Theorem \ref{Theorem Bound Empirical Measure Convergence}.
\begin{proof}
Thanks to the Triangle Inequality,
\begin{align}
\sup_{t\leq T} d_W\big( \hat{\mu}^n_t , \mu_t \big)  \leq  \sup_{t\leq T}  d_W\big( \tilde{\mu}^n_t , \mu_t \big) +  \sup_{t\leq T}  d_W\big( \tilde{\mu}^n_t , \hat{\mu}^n_t \big) .
\end{align}
Thanks to Lemma \ref{Lemma Law of Large Numbers}, it suffices that we prove that $\mathbb{P}$-almost-surely,
\begin{align}
\lim_{n\to\infty}  \sup_{t\leq T}  d_W\big( \tilde{\mu}^n_t , \hat{\mu}^n_t \big) =& 0 \text{ and } \label{eq: second difference of empirical measures} \\
\lim_{n\to\infty} \sup_{t\leq T} \ell(t) =& 0.
\end{align}
Now Jensen's Inequality implies that
\begin{align}
d_W\big( \hat{\mu}^n_t , \tilde{\mu}^n_t  \big) \leq & \bigg( n^{-1}\sum_{j\in I_n}\bigg\lbrace (\bar{y}^j_{e,t} - y^j_{e,t})^2  + (\bar{y}^j_{i,t} - y^j_{i,t})^2 \bigg\rbrace \bigg)^{1/2} \nonumber \\
=& \sqrt{q_t}.
\end{align}
It thus suffices that we prove that
\begin{align}
\lim_{n\to\infty}  \sup_{t\leq T} q_t =& 0 \text{ and } \label{eq: second difference of empirical measures two} \\
\lim_{n\to\infty} \sup_{t\leq T} \ell(t) =& 0. \label{eq: ell t to zero}
\end{align}
Define the stopping time
\[
\phi = \inf \big\lbrace t \leq T \; : \; \sqrt{q_t} = \epsilon_T \big\rbrace .
\]
 It follows from Corollary \ref{Corollary Bound ell t} and Lemma \ref{Lemma Bound q t n quarter} that $\mathbb{P}$-almost-surely, for all $t\leq \phi$,
 for all large enough $n$,
 \begin{align}
 \sqrt{\ell_t} &\leq  n^{-1/4} +  2 \frac{c_T}{\xi_T} \sqrt{q_t} \\
 q_t &\leq C n^{-1/4} .
 \end{align}
We thus find that for all $t\leq \phi$,
\begin{align}
 \sqrt{\ell_t} \leq n^{-1/4} +  2\sqrt{C} \frac{c_T}{\xi_T} n^{-1/8}.
\end{align} 
Thus for all large enough $n$, it holds that $\phi  = T$ and \eqref{eq: second difference of empirical measures two} and \eqref{eq: ell t to zero} both hold.

\end{proof}
 
 \subsection{Proof Details} \label{Subsection Proof Details}
 
 We outline the proof of Lemma \ref{Lemma Bound the frowth of ell t}.
  \begin{proof}
Using the expressions in \eqref{eq: dv et} and \eqref{eq: dv it}, thanks to Ito's Lemam it holds that
\begin{multline}
 d\ell_t   =   2\sum_{a=1}^M \sum_{\alpha \in \lbrace e,i \rbrace} \bigg( v^a_{\alpha,t} - \bar{v}^a_{\alpha,t} \bigg)\bigg\lbrace - \frac{d\bar{v}^a_{\alpha,t}}{dt} + n^{-1} \sum_{j\in I_n} \sum_{b=1}^M \mathcal{Q}^{-1}_{ab} h_b(x^j_n)\bigg\lbrace f_\alpha(z^j_{\alpha,t})  \\ + n^{-1/2} \sum_{k\in I_n} \bigg(  \mathcal{K}_{ee}(x^j_n , x^k_n) G_{ee}(z^k_{e,t}) - \mathcal{K}_{ei}(x^j_n , x^k_n)  G_{ei}(z^k_{i,t}) \bigg) \bigg\rbrace  \bigg\rbrace dt  \\
  +2\sum_{a=1}^M \sum_{\alpha\in \lbrace e,i \rbrace} \big( v^a_{\alpha,t} - \bar{v}^a_{\alpha,t} \big) d\bar{w}^a_{\alpha,t}  
  + \sum_{a=1}^M \sum_{\alpha\in \lbrace e,i \rbrace} Q^a_{\alpha,t} dt,
  \end{multline}
  where $Q^a_{\alpha,t}$ is the derivative of the quadratic variation of $\bar{w}^a_{\alpha,t}$. We assume that $t$ is such that
  \begin{align}
\sup_{r\leq t} d_W\big( \hat{\mu}^n_r , \mu_r \big) &\leq \epsilon_T \text{ and } \label{eq: ell bound cond 1 1222}\\
\sup_{r\leq t} \ell_r &\leq \epsilon_T. \label{eq: ell upperbound temporary first}
\end{align} 
 Thanks to the Mean-Value Theorem, we find that for some $\lambda_t \in [0,1]$, writing  $\tilde{v}(t) = \lambda v(t) + (1-\lambda) \bar{v}(t)$, it holds that for $a\leq M$ and $\alpha\in \lbrace e,i \rbrace$,
 \begin{align} \label{eq: bound G difference 1}
 \mathcal{G}^a_{\alpha}\big(v_e(t),v_i(t),\hat{\mu}^n_t\big) -  \mathcal{G}^a_{\alpha}\big(\bar{v}_e(t) ,\bar{v}_i(t),\hat{\mu}^n_t\big) = \sum_{b=1}^M \sum_{\beta \in \lbrace e,i \rbrace} \mathcal{J}^{ab}_{\alpha\beta}\big( \tilde{v}_e(t), \tilde{v}_i(t) , \hat{\mu}^n_t\big)\big( v^b_{\beta}(t) - \bar{v}^b_{\beta}(t) \big)
 \end{align}

Now, the eigenvalues of a matrix depend continuously on its entries. Since the eigenvalues of $\mathcal{J}( \bar{v}_e(t),\bar{v}_i(t) ,\mu_t)$ must have real component less than or equal to $-\xi_T$, as long as $\epsilon_T$ is small enough it must hold that the real components of the eigenvalues of $\mathcal{J}( \tilde{v}_e(t),\tilde{v}_i(t),\hat{\mu}^n_t)$ are less than or equal to $-\xi_T / 2$.  Furthermore, since $ \mathcal{G}^a_{\alpha}\big(\bar{v}_e(t) ,\bar{v}_i(t),\mu_t\big) = 0$, it holds that
 \begin{align}
\big| \mathcal{G}^a_{\alpha}\big(\bar{v}_e(t) ,\bar{v}_i(t),\hat{\mu}^n_t\big) \big| =&  \big| \mathcal{G}^a_{\alpha}\big(\bar{v}_e(t) ,\bar{v}_i(t),\hat{\mu}^n_t\big) -  \mathcal{G}^a_{\alpha}\big(\bar{v}_e(t) ,\bar{v}_i(t),\mu_t\big) \big| \nonumber \\
\leq & c_T d_W\big( \hat{\mu}^n_t , \mu_t \big) \label{eq:G difference bound final}
 \end{align}
 for some constant $c_T > 0$. It thus follows from \eqref{eq: bound G difference 1} -\eqref{eq:G difference bound final} that there exists a constant $c_T$ such that
 \begin{multline}
    \frac{2}{\sqrt{n}}\sum_{a=1}^M \sum_{\alpha \in \lbrace e,i \rbrace} \bigg( v^a_{\alpha,t} - \bar{v}^a_{\alpha,t} \bigg) \sum_{k\in I_n} \bigg(  \mathcal{K}_{ee}(x^j_n , x^k_n) G_{ee}(z^k_{e,t}) - \mathcal{K}_{ei}(x^j_n , x^k_n)  G_{ei}(z^k_{i,t}) \bigg)\\
    \leq \sqrt{n}c_T \sqrt{\ell(t)} d_W\big( \hat{\mu}^n_t , \mu_t \big) -\xi_T \frac{\sqrt{n}}{2} \ell(t). 
 \end{multline}

 It is immediate from the expressions in Corollary \ref{Corollary Autonomous Dynamics} (which is in turn an immediate corollary of the existence results in Lemma \ref{Lemma Existence of Hydro Limit}) that
 \begin{align}
\sup_{t\leq T} \bigg| \frac{d\bar{v}^a_{\alpha,t}}{dt} \bigg| < \infty .
 \end{align} 
 It follows immediately from our assumption that the functions are upperbounded that
 \begin{align}
 \bigg| n^{-1} \sum_{j\in I_n} \sum_{b=1}^M \mathcal{Q}^{-1}_{ab} h_b(x^j_n) f_e(z^j_{e,t}) \bigg| &\leq \rm{Const} \\
 \bigg| n^{-1} \sum_{j\in I_n} \sum_{b=1}^M \mathcal{Q}^{-1}_{ab} h_b(x^j_n) f_i(z^j_{i,t}) \bigg| &\leq \rm{Const} .
 \end{align} 
The quadratic variation of $\bar{w}^a_{e,t}$ upto time $t$ can be seen to be
\begin{align}
n^{-2}\sum_{j\in I_n} \int_0^t \bigg(  \sum_{b=1}^M \mathcal{Q}^{-1}_{ab} h_b(x^j_n) \sigma(z^j_{e,s}) \bigg)^2 ds,
\end{align}
which is upperbounded by $\rm{Const} \times    t / n$, since $|h_b|$ and $|\sigma_e|$ are upperbounded. The quadratic variation of $\bar{w}^a_{i,t}$ is also upperbounded by $\rm{Const} \times    t / n$ for the same reason. Finally, the quadratic variation of
\begin{align}
2\sum_{a=1}^M \sum_{\alpha\in \lbrace e,i \rbrace} \big( v^a_{\alpha,t} - \bar{v}^a_{\alpha,t} \big) d\bar{w}^a_{\alpha,t} 
\end{align}
is upperbounded by 
\[
n^{-1} c_T \int_0^t \ell(s) ds.
\]
Combining the above estimates, we obtain the Lemma.
 \end{proof}
 
 We next prove Lemma \ref{Lemma Bound q t n quarter}.
\begin{proof}
It follows from Ito's Lemma that
\begin{multline*}
dq_t =  \frac{2}{n} \sum_{j\in I_n}\sum_{\alpha\in \lbrace e,i \rbrace} u^j_{\alpha,t}\bigg(f_\alpha(z^{j}_{\alpha,t}) - f_\alpha(\bar{z}^{j}_{\alpha,t}) \\ - n^{-1} \sum_{k\in I_n} \sum_{a,b = 1}^{M} Q^{-1}_{ab} h_a(x^j_{n})h_b(x^k_{n})\bigg\lbrace f_\alpha(z^{k}_{\alpha,t}) - f_\alpha(\bar{z}^{k}_{\alpha,t})  \bigg\rbrace \bigg) dt \\
+\frac{1}{n} \sum_{j\in I_n} \sum_{\alpha\in \lbrace e,i \rbrace} \bigg(\sigma_\alpha(x^j_n , z^{j}_{\alpha,t}) -  \sigma_\alpha(x^j_n, \bar{z}^{j}_{\alpha,t}) \bigg)^2 dt
+ \frac{2}{n}\sum_{j\in I_n}\sum_{\alpha\in \lbrace e,i \rbrace}  u^j_{\alpha,t} \bigg(\sigma_\alpha(x^j_n , z^{j}_{\alpha,t}) -  \sigma_\alpha(x^j_n, \bar{z}^{j}_{\alpha,t}) \bigg) dW^j_{\alpha,t}.
\end{multline*}
Employing the fact that (i) the functions are globally Lipschitz and (ii) the Cauchy-Schwarz Inequality, we find that there exists a constant $C > 0$ (this constant is chosen independently of $n$) such that
\begin{equation}
dq_t \leq C\bigg(   q_t +      \| v_t - \bar{v}_t \|^2 \bigg) dt+\frac{2}{n}\sum_{j\in I_n}\sum_{\alpha\in \lbrace e,i \rbrace}  u^j_{\alpha,t} \bigg(\sigma_\alpha(x^j_n , z^{j}_{\alpha,t}) -  \sigma_\alpha(x^j_n , \bar{z}^{j}_{\alpha,t}) \bigg) dW^j_{\alpha,t}.
\end{equation}
Define 
\begin{equation}
r_t =  q_t^{1/2} .
\end{equation}
The function $x \mapsto \sqrt{x}$ has a negative second derivative, and thus by Ito's Lemma we find that
\begin{align}
dr_t \leq \bigg( C r_t +  C  \| v_t - \bar{v}_t \| r_t^{-1}   \bigg) dt + d\grave{w}_t , \label{eq: r dt inequality}
\end{align}
where $\grave{w}_t$ is an $\mathcal{F}$-adapted martingale with quadratic variation 
\begin{align}
o_t =& \frac{1}{n^2}\sum_{j\in I_n} \sum_{\alpha\in \lbrace e,i \rbrace}  \int_0^t q_s^{-1}  (u^j_{\alpha,s})^2 \bigg(\sigma_\alpha(x^j_n , z^{j}_{\alpha,s}) -  \sigma_\alpha(x^j_n , \bar{z}^{j}_{\alpha,s}) \bigg)^2  ds.
\end{align}
Since $\sigma$ is uniformly upperbounded by $C_{\sigma}$, it holds that for all $t \leq T$,
\begin{align}
o_t \leq  \frac{4t}{n}C_{\sigma}^2.
\end{align}
Using the time-rescaled representation of a Martingale \cite{Karatzas1991}, we thus find obtain that for a standard Brownian Motion $w(t)$,
\begin{align}
\mathbb{P}\bigg( \sup_{\zeta \leq t\leq \iota}\big| \grave{w}_{t} - \grave{w}_{\zeta} \big| \geq n^{-1/4} \bigg) \leq \mathbb{P}\bigg( \sup_{ t\leq T}\big| w\big( 4t C_{\sigma}^2 / n \big) \big| \geq n^{-1/4} \bigg) \\
\leq \rm{Const} \exp\bigg(  - \rm{Const} n^{1/2}  \bigg)
\end{align}
thanks to Doob's Submartingale Inequality. Furthermore, if
\begin{align}
  \sup_{\zeta \leq t\leq \iota}\big| \grave{w}_t - \grave{w}_{\zeta} \big| \leq n^{-1/4} \label{eq: grave w t bound}
  \end{align}
  then necessarily for all $t\leq \zeta_n$, there is a constant $\tilde{C}$ such that
\begin{equation}
 r_t \leq  \int_0^t \bigg( C r_s +  \tilde{C} n^{-1/4} + \tilde{C} r_s  \bigg) ds +  n^{-1/4}. \label{eq: r dt inequality 3}
\end{equation}
It follows from an application of Gronwall's Inequality to \eqref{eq: r dt inequality 3} that for all $t\leq \zeta_n$, if \eqref{eq: grave w t bound} holds, then
\begin{align}
\sup_{t \leq T} r_t \leq \rm{Const}\times n^{-1/4}
\end{align}
This implies the Lemma.
\end{proof}
   
 We next state the proof of Corollary \ref{Corollary Bound ell t}.
\begin{proof}
Let 
\begin{align}
\tau = \inf\bigg\lbrace t\geq 0 \; : \;   \ell_t = \epsilon_T \text{ or }d_W\big( \hat{\mu}^n_t , \mu_t \big) = \epsilon_T  \bigg\rbrace .
\end{align}
We must show that
\begin{align}
\lsup{n} n^{-1/3} \log \mathbb{P}\bigg( \text{ For some }t\leq \tau, \text{ it holds that }\sqrt{\ell_t} >n^{-1/4} +  4 \frac{c_T}{\xi_T} d_W\big( \hat{\mu}^n_t , \mu_t \big) \bigg) < 0.
\end{align}
 Let $W(t)$ be a standard $\mathbb{R}$-valued Brownian Motion. Using the time-rescaled representation of a continuous Martingale to find that
 \begin{align}
  \mathbb{P}\bigg( \sup_{t\leq \tau \wedge T} \big| \tilde{w}_t \big|   \geq n^{-1/3}/2 \bigg) &\leq  2\mathbb{P}\bigg( \sup_{t\leq T} W\big( n^{-1}  c_T t \epsilon_T \big)   \geq n^{-1/3} /2 \bigg) \nonumber \\
  &\leq \exp\big( - n^{-1/3} \tilde{c} \big), \label{eq: n minus third c}
 \end{align}
 for some constant $\tilde{c} > 0 $ that is independent of $n$. We henceforth assume that
 \begin{align}
 \sup_{\iota \leq t\leq \zeta \wedge T \wedge \tau} \big| \tilde{w}_t - \tilde{w}_{\iota}\big|   \leq    n^{-1/3} . \label{eq: henceforth eq}
 \end{align}
 It therefore follows from Lemma \ref{Lemma Bound the frowth of ell t} that for all $\tilde{t} < t\leq T\wedge \tau$,
\begin{equation}
\ell_t  \leq \ell_{\tilde{t}}+ 2n^{-1/3} 
+ \int_{\tilde{t}}^t \bigg\lbrace \sqrt{\ell}_s c_T   + \sqrt{n}\bigg( c_T\sqrt{\ell_s} d_W\big( \hat{\mu}^n_s , \mu_s \big)  - \frac{  \xi_T}{2} \ell_s \bigg)  \bigg\rbrace ds . \label{eq: intermediate ell t ell tilde t inequality}
\end{equation}
Suppose for a contradiction that there exists $\iota \leq \tilde{t} < t \leq T\wedge \tau$ such that for all $s \in [\tilde{t}, t]$, it holds that
\begin{align}
\frac{\xi_T}{2}\ell_{\tilde{t}} &= 2  c_T \sqrt{\ell_{\tilde{t}}} d_W\big( \hat{\mu}^n_{\tilde{t}} , \mu_{\tilde{t}} \big) + \frac{1}{4} n^{-1/4} \label{eq: final to contradict 1}\\
\frac{\xi_T}{2} \ell_{s} & > 2 c_T  \sqrt{\ell_{s}}  d_W\big( \hat{\mu}^n_s , \mu_s \big) +\frac{1}{4} n^{-1/4} \\
\frac{\xi_T}{2}  \ell_{t} &= 2 c_T  \sqrt{\ell_{t}} d_W\big( \hat{\mu}^n_t , \mu_t \big) +  n^{-1/4} \label{eq: final to contradict 3}
\end{align}
It then follows from \eqref{eq: intermediate ell t ell tilde t inequality} that
\begin{align}
 \ell_t  -  \ell_{\tilde{t}}  \leq &  2n^{-1/3}+ c_T\int_{\tilde{t}}^t \sqrt{\ell_s}  ds - \frac{  \epsilon_T}{2}(t-\tilde{t})n^{1/4} \\
\leq & 2n^{-1/3} + c_T \epsilon_T( t- \tilde{t}) - \frac{  \epsilon_T}{2}(t-\tilde{t})n^{1/4} \label{eq: final to contradict 4}
\end{align}
since for all $t \leq \tau$, $\sqrt{\ell_s}  \leq \epsilon_T$. It also follows from \eqref{eq: final to contradict 1} - \eqref{eq: final to contradict 3} that
\begin{align}
\ell_{t} - \ell_{\tilde{t}} = \frac{3}{4} n^{-1/4}. \label{eq: final to contradict 5}
\end{align}
However once $n$ is large enough that $2n^{-1/3} < \frac{3}{4} n^{-1/4}$ and $n^{1/4} / 2 > c_T$, then \eqref{eq: final to contradict 4} contradicts \eqref{eq: final to contradict 5}. We thus find that, for all large enough $n$, it is impossible that \eqref{eq: final to contradict 1}-\eqref{eq: final to contradict 3} hold, as well as \eqref{eq: henceforth eq}. Since the stochastic processes are all continuous-in-time, it is therefore impossible that there exists $t\leq T \wedge \tau$ such that \eqref{eq: final to contradict 4} holds. We thus find that
\begin{multline}
\lsup{n} n^{-1/3} \log \mathbb{P}\bigg( \text{ For some }t\leq \tau, \text{ it holds that }\sqrt{\ell_t} >n^{-1/4} +  4 \frac{c_T}{\xi_T} d_W\big( \hat{\mu}^n_t , \mu_t \big) \bigg) \\
\leq \lsup{n} n^{-1/3} \log \mathbb{P}\bigg( \sup_{\iota \leq t\leq \zeta \wedge T \wedge \tau} \big| \tilde{w}_t - \tilde{w}_{\iota}\big|   >    n^{-1/3} \bigg) \\
\leq - \rm{Const},
\end{multline}
thanks to \eqref{eq: n minus third c}.
\end{proof}

We finish with a proof of Lemma \ref{Lemma Existence of Hydro Limit}.
\begin{proof}
We start by proving that there exists a unique solution to the dynamical system in Corollary \ref{Corollary Autonomous Dynamics}. It follows immediately from this that for all $1\leq a \leq M$ and $\alpha \in \lbrace e,i \rbrace$,
  \begin{align}
 \frac{d}{dt}\mathcal{G}_{\alpha}^a( \bar{v}_t, \mu_t ) = 0,
 \end{align}
 which means that any solution $\bar{v}_t$ must also satisfy the requirements of Lemma \ref{Lemma Existence of Hydro Limit}.

For any $\bar{v} \in \mathcal{C}\big( [0,T], \mathcal{C}_M(\mathcal{E}) \big)$, let $\Psi(\bar{v}) \in \mathcal{C}\big( [0,T], \mathcal{P}\big( \mathbb{R}^2 \big) \big)$ be such that $\Psi(\bar{v}) := \big( \Psi_t(\bar{v}) \big)_{t\in [0,T]}$, and $\Psi_t(\bar{v})$ is the law of $(y_{e,t}, y_{i,t})$. We let $d_T$ be the metric on $ \mathcal{C}\big( [0,T], \mathcal{P}\big( \mathbb{R}^2 \big) \big)$ such that
\begin{align}
d_T\big( \mu_{[0,T]} , \nu_{[0,T]} \big) = \sup_{t\leq T} d_W\big( \mu_t, \nu_t \big),
\end{align}
and $d_W$ is the Wasserstein distance on $\mathcal{P}(\mathbb{R}^2)$. It is straightforward to show using Gronwall's Lemma that there exists a constant $C> 0$ such that
\begin{align}
d_T\big( \Psi(\bar{v}) , \Psi(\tilde{v}) \big) \leq C \sup_{t\leq T} \| \bar{v} - \tilde{v} \|,
\end{align}
where
\begin{align}
\| \bar{v} - \tilde{v} \| = \sup_{1\leq a \leq M} \bigg| \int_{\mathcal{E}} h_a(x)\big(\bar{v}_e(x) - \tilde{v}_e(x) \big) \kappa(dx) \bigg|
+ \sup_{1\leq a \leq M} \bigg| \int_{\mathcal{E}} h_a(x)\big(\bar{v}_i(x) - \tilde{v}_i(x) \big) \kappa(dx) \bigg|.
\end{align}
We next define, for $1\leq a \leq M$ and $\alpha \in \lbrace e,i \rbrace$,
\begin{align}
\Gamma_{\alpha,t}^a &: \mathcal{C}\big( [0,T], \mathcal{C}_M(\mathcal{E}) \big)^2  \mapsto \mathbb{R} \\
\Gamma_{\alpha,t}^a\big(\bar{v}_e, \bar{v}_i\big) &=  - \sum_{b=1}^M \sum_{\beta\in \lbrace e,i \rbrace} \mathcal{L}^{ab}_{\alpha\beta} \big( \bar{v}_e(t), \bar{v}_i(t) , \Psi_t(\bar{v}) \big) \mathcal{H}_{\beta}^b\big( \bar{v}_e(t), \bar{v}_i(t) ,  \Psi_t(\bar{v}) \big). 
\end{align}
We need to prove that there exists a solution to the $2M$ equations, for $\alpha\in \lbrace e,i \rbrace$ and $1\leq a \leq M$,
\begin{align}
\frac{d\bar{v}^a_{\alpha}}{dt} = \Gamma_{\alpha,t}^a(\bar{v}) \label{eq: to show solution Picard}
\end{align}
Since every function in the definition of $\Gamma_{\alpha,t}^a$ is uniformly bounded, there exists a global constant $c_{\delta} > 0$ such that as long as for all $t\leq T$ it holds that
\begin{align}
\det\big( \mathcal{J}(\bar{v}_e(t), \bar{v}_i(t), \Psi_t(\bar{v}) \big) \geq \delta .
\end{align}
then necessarily
\begin{align}
\sup_{\alpha\in \lbrace e,i \rbrace} \sup_{1\leq a \leq M}\big| \Gamma_{\alpha,t}^a\big(\bar{v}_e, \bar{v}_i\big) - \Gamma_{\alpha,t}^a\big(\tilde{v}_e, \tilde{v}_i\big)\big| \leq c_{\delta} \sup_{s\leq t}\| \bar{v}_s - \tilde{v}_s \|. \label{eq: Lipschitz Picard}
\end{align}
It now follows using Picard Iterations that, for any $\delta > 0$, there is a unique solution to \eqref{eq: to show solution Picard} upto the first time $T$ that 
\begin{align}
\det\big( \mathcal{J}(\bar{v}_e(T), \bar{v}_i(T), \Psi_T(\bar{v}) \big) = \delta .
\end{align}
Write such $T := T_{\delta}$. We next claim that $\eta = \sup_{\delta > 0}T_{\delta}$. It is immediate that if $\eta < \infty$, then necessarily 
\[
\lim_{T\to\eta^-} \det\big( \mathcal{J}(\bar{v}_T , \mu_T) \big) = 0.
\]
Thus $\eta$ is the maximal time that a balanced solution can exist.
\end{proof}

\textit{Acknowledgements:}

Thanks to Rainer Engelken (Columbia University) for a helpful discussion and for supplying some references.
 
 \bibliographystyle{plain}
\bibliography{library,balancedbib}

\end{document}